\documentclass[11pt,a4paper,reqno]{amsart}

\usepackage{a4wide,color}
\usepackage{cancel}

\usepackage{graphics,color}
\usepackage{amsmath}
\usepackage{amssymb}
\usepackage{mathrsfs}
\usepackage{cite}
\usepackage{verbatim}
\usepackage{float}
\usepackage{graphicx}
\usepackage{amsthm}
\usepackage{textcomp}
\usepackage{subfig}
\usepackage[colorlinks=true,linkcolor=blue,citecolor=red]{hyperref}

\newif\ifdraft
\draftfalse 

\newcommand{\A}{\mathcal A}
\newcommand{\jump}[1]{\text{{\rm \textlbrackdbl}}{#1}\text{{\rm \textrbrackdbl}}}

\newcommand{\nada}[1]{}

\newcommand{\Om}{\Omega}

\newcommand{\R}{\mathbb{R}}
\newcommand{\res}{\mathop{\hbox{\vrule height 7pt width 0.5pt depth 0pt
\vrule height 0.5pt width 6pt depth 0pt}}\nolimits}

\def\Su{\mathbb S^1}

\def\d{\mathrm{d}}
\def\mul{\mathrm{mul}}
\def\N{\mathbb{N}}

\numberwithin{equation}{section}
\mathchardef\emptyset="001F

\newtheorem{theorem}{Theorem}[section]
\newtheorem{definition}[theorem]{Definition}
\newtheorem{prop}[theorem]{Proposition}
\newtheorem{cor}[theorem]{Corollary}
\newtheorem{lemma}[theorem]{Lemma}

\theoremstyle{definition}
\newtheorem{remark}[theorem]{Remark}

\title{On the singular planar Plateau problem
}

\author{Marco Caroccia}
\address{Politecnico di Milano,  Dipartimento di matematica,  Piazza Leonardo da Vinci,  32,  20133 Milano MI}
\email{marco.caroccia@polimi.it}

\author{Riccardo Scala}
\address{Dipartimento di Ingegneria dell'Informazione e Scienze Matematiche, Universit\`a di Siena, 53100 Siena, Italy.}
\email{riccardo.scala@unisi.it}

\begin{document}
	
	\begin{abstract}
		Given any $\Gamma=\gamma(\Su)\subset\R^2$, image of a Lipschitz curve $\gamma:\Su\rightarrow \R^2$,  not necessarily injective,  we provide an explicit formula for computing the value of  
\[
\mathcal A(\gamma):=\inf\left\{\left. \int_{B_1(0)}|\det(\nabla u)| \d x \ \right| \ u=\gamma \text{ on }\Su\right\},
\]
where the infimum is computed among all Lipschitz maps $u:B_1(0)\rightarrow \R^2$ having boundary datum $\gamma$.  This coincides with the area of a minimal disk spanning $\Gamma$, i.e., a solution of the  Plateau problem of disk type.  The novelty of the results relies in the fact that we do not assume the curve $\gamma$ to be injective and our formula allows for countably many self-intersections.
	\end{abstract}

\maketitle
	
\section{Introduction}
In this paper,  we consider the Plateau problem for singular curves, i.e., for curves $\gamma$ whose image is not necessarily a Jordan curve, but might have self-intersections.  Specifically,  let us denote by $D:=\{x\in \R^2 \ | \ |x|\leq 1\}$ the unit disk centered at $0$ in $ \R^2$ and by $\Su:=\partial D$.  We focus on planar singular curves, and we consider the following Plateau problem: given a Lipschitz curve $\gamma \in  C^{0,1} (\Su;\R^2)$, compute the value of the infimum 
\begin{equation}\label{eqn:InfProbl}
\mathcal A(\gamma)=\inf\left\{A(\psi)\ \left| \ \psi\in C^{0,1}( D;\R^2)\text{ such that } \psi=\gamma\text{ on }\Su \right. \right\}
\end{equation}
with
\[
A(\psi):=\int_D |\det (\nabla \psi)| \d x.
\]
We do not focus on existence and regularity of solutions $\psi$ of the singular Plateau problem \eqref{eqn:InfProbl} but our scope is to understand how the value of $\A(\gamma)$ is related to the geometry of the curve $\gamma$, and in some cases how to compute it.  Questions regarding the existence and regularity of solutions for this problem have been first investigated in \cite{Hass91} in 1991.  More recently, following the analysis in \cite{Creutz} (see also \cite{CF} for a metric space setting), it is shown that a solution of the singular Plateau problem can be chosen in suitable Sobolev spaces.  \\

When $\Gamma:=\gamma(\Su)$ represents a Jordan curve then \eqref{eqn:InfProbl} corresponds to the classical Plateau problem,  which consists into finding a map $\psi\in H^1(\text{{\rm int}}(D);\R^3)\cap C^0(D;\R^3)$ (here $\mathrm{int}(D)$ stands for the \textit{interior} of $D$) minimizing the area functional
\begin{align}\label{area_functional}
	A(\psi)=\int_D|\partial_{x_1}\psi\wedge\partial_{x_2}\psi|d x.
\end{align}
among all maps whose restrictions on $\Su$ coincide,  up to reparametrization,  with $\gamma$.
 We refer to \cite{DHS} for a general treatment of the parametric approach to the aforementioned Plateau problem.  Here we just recall some well-known facts that are useful for our discussion.  For instance, it is well-known that if $\psi$ is a minimizer of \eqref{area_functional}, for $\Gamma$ being a Jordan curve,  then it also turns out to be harmonic and conformal in the interior of $D$ (cf.  \cite{DHS}).
In the special case where the Jordan curve is also planar,  as in $\gamma:\Su\rightarrow \R^2$,  the solution to the Plateau problem is provided by the Riemann map $\psi$,  i.e., a bi-holomorphic bijection between  $\text{{\rm int}}(D)$ and $S$,  the unique open bounded connected component of $\R^2\setminus\Gamma$ (see Proposition \ref{Riemann_plateau} below).  Since $\psi(D)=S$,  it follows that 
$\mathcal A(\gamma)=|S|$.   In very few other cases, it is possible to find the exact value of $A(\psi)$,  when $\psi$ is a solution to the Plateau problem.  However,  unless $\psi$ is not explicit, or unless $\Gamma$ has some special geometry, $A(\psi)$ is not known.  We are here able to provide two very general results (Theorem \ref{teo1_intro}, \ref{teo2_intro}) for computing the value of $\A(\gamma)$,   enhancing the role of the geometry of $\gamma$.  Our results apply to general Lipscthitz curve (with finite length) that might possibly have up to countably many self-intersections.  \\

We start our analysis by focusing on curves $\gamma$ with the following property:
\medskip

\begin{itemize}
\item[(P)] Letting $\Gamma:=\gamma(\Su)$, the set $\R^2\setminus \Gamma$ has finitely many connected components.
\end{itemize}
\medskip

Under hypothesis (P),  if $n$ denotes the number of bounded connected components of $\R^2\setminus \Gamma$,  then we fix $n$ points $P_i$, $i=1,\dots,n$,  each belonging to exactly one of these connected components denoted by $U_i$ (see Figure \ref{fig1:NotationAndExample}).  We now consider $\gamma$ as a loop in $\R^2\setminus \{P_1,\dots,P_n\}$,  with starting point $P:=\gamma((1,0))$.  So there is a representative of $\gamma$ in the homotopy group based at $P$,  $F(n):=\pi_1(\R^2\setminus \{P_1,\dots,P_n\})$.  We denote this representative by $\gamma$ itself.  Here and in what follows, for the sake of clarity,  we will always omit to specify the dependence of $\pi_1$ from $P:=\gamma((0,1))$.  Our first main result can then be summarized in the following theorem:

\begin{theorem}\label{teo1_intro}
	Let $\gamma:\Su\rightarrow \R^2$ be a Lipschitz curve satisfying Property (P), and let $n\geq0$ be the number of bounded connected components of $\R^2\setminus \Gamma$. Then there is a function $f=f^\gamma:F(n)\rightarrow \mathbb N^n$ such that the following holds:
	\begin{align}
\mathcal A(\gamma)=\sum_{i=1}^nf_i(\gamma)\mathcal L^2(U_i), 
	\end{align}
where $U_i$ is the $i$-th bounded connected components of $ \R^2\setminus \Gamma$, and $\mathcal L^2(U_i)$ denotes its Lebesgue measure.
\end{theorem}

The function $f^\gamma$ is explicit. The precise value of $\mathcal A(\gamma)$ is detailed in Theorem \ref{teo_main}.  We now provide a brief explanation on computing $f^{\gamma}$ aside from the technical details discussed in Section \ref{sct:PfMain}.  Since $F(n)$ is the free group on $n$ elements,  here $\gamma$ is equivalent to a finite sequence of symbols in $\{\sigma_1,\dots,\sigma_n\}\cup\{\sigma_1^{-1},\dots,\sigma_n^{-1}\}$, the generators of $F(n)$. 
For each word $\eta$ in $F(n)$, we associate it with specific $n$-tuples of natural numbers $(k_1,\ldots,k_n)\in \mathbb{N}^n$. This $n$-tuple is built by counting how many times each element $\sigma_i$ or $\sigma_i^{-1}$ must be erased from $\eta$ to obtain the null word.  For example the word $\sigma_1\sigma_2$ is associated with the $n$-tuple $(1,1,0,\ldots,0)$. This is because the null word can be obtained by erasing one $\sigma_1$ and one $\sigma_2$.   The word $\sigma_1\sigma_2\sigma_1^{-1}$ instead becomes the null word by erasing either a single $\sigma_2$,  or by erasing one each of $\sigma_1$, $\sigma_1^{-1}$, and $\sigma_2$.  Thus both the $n$-tuple $(2,1,0,\ldots,0)$ and $(0,1,0,\ldots,0)$ are associated to the word $\sigma_1\sigma_2\sigma_1^{-1}$.  We gather all $n$-tuples related to any word representing $\gamma$ into a set, denoted as $\mathrm{Ad}(\gamma)\subset \N^n$.  Let us point out to the reader that the set $\mathrm{Ad}(\gamma)$,  formally built in Subsection \ref{sbsct:MainResult},  is created by considering \textit{injections} rather than cancellations.  An injection is the insertion of the inverse $\sigma_{i}^{-1}$ in the word,  to annihilate a $\sigma_{i}$.  Actually,  due to technical reasons, a slightly more general operation is needed to construct $\mathrm{Ad}(\gamma)$ accurately.  Once this set is built we show that we can compute the value $\mathcal{A}(\gamma)$ as
\begin{equation}\label{eqn:minPr}
\mathcal{A}(\gamma):=\min\left\{\left. \sum_{i=1}^n k_i |U_i| \ \right| \ (k_1,\ldots,k_n)\in \mathrm{Ad}(\gamma)\right\}
\end{equation}
and this is the content of Section \ref{sct:PfMain}, where Theorem \ref{teo_main} is proved.  Thus $f^{\gamma}$ is just the $n$-tuple achieving the minimum in \eqref{eqn:minPr}.  For instance,  considering the curve in Figure \ref{fig1:NotationAndExample} we have $\gamma=\sigma_1\sigma_2^{-1}\sigma_3$ and thus the minimal way to obtain the null word is by erasing each single generator $\sigma_1,\sigma_2,\sigma_3$.  Thus $(1,1,1)\in \mathrm{Ad}(\gamma)$ and 
\[
\A(\gamma)=|U_1|+|U_2|+|U_3|.
\]
In Figure \ref{fig2:ExamplefFiore} instead we have that $\gamma\equiv \sigma_1\sigma_2\sigma_3\sigma_2^{-1}$.  Thus, by erasing $\sigma_3$ from the word, we obtain $\gamma'=\sigma_1\sigma_2 \sigma_2^{-1}\equiv \sigma_1$, and then by erasing $\sigma_1$, we get the null word. Thus $(1,0,1)\in \mathrm{Ad}(\gamma)$ and it is immediate to see that it is the optimum: 
\[
\mathcal{A}(\gamma)=|U_1|+|U_3|.
\]
\begin{figure}[t!]
\begin{center}
\includegraphics[scale=0.6]{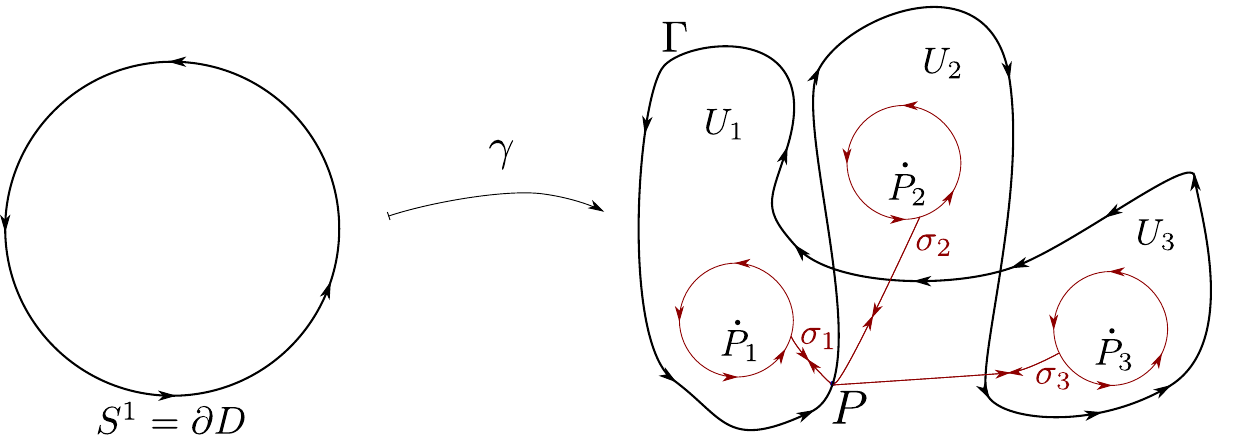}
\caption{A non-injective curve with a support set $\Gamma$, such that $\mathbb{R}^2\setminus \Gamma$ consists of three bounded connected components. In terms of equivalence classes in $\pi_1$, we have $\gamma\equiv \sigma_1\sigma_2^{-1}\sigma_3$ in $\pi_1(\mathbb{R}^2\setminus{P_1,P_2,P_3})$.}\label{fig1:NotationAndExample}
\end{center}
\end{figure}

\begin{figure}[t!]
\begin{center}
\includegraphics[scale=0.6]{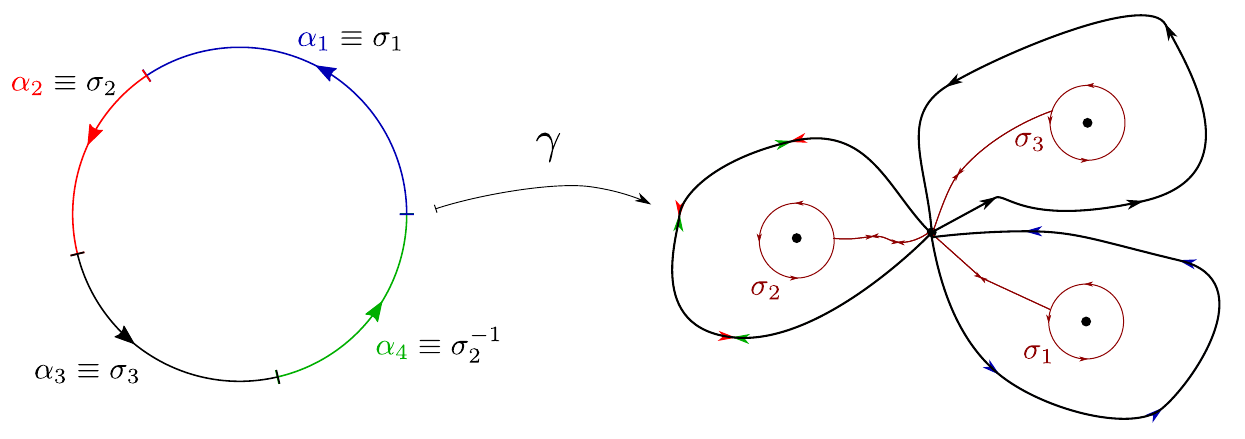}
\caption{A non-injective curve whose support $\Gamma$ is such that $\R^2\setminus \Gamma$ has $3$ bounded connected components.  As equivalence class we have $\gamma\equiv \sigma_1\sigma_2\sigma_3\sigma_2^{-1}$ in $\pi_1(\R^2\setminus\{P_1,P_2,P_3\})$ and thus $(1,0,1)\in \mathrm{Ad}(\gamma)$.}\label{fig2:ExamplefFiore}
\end{center}
\end{figure}

Note that annihilation is not mandatory for every pair $\sigma_i, \sigma_i^{-1}$: consider a curve supported on $\Gamma$ and splitting $\R^2\setminus \Gamma$ into $2$ connected components with $|U_1| > |U_2|$.  Let $\gamma$ be its parametrization, where $\gamma\equiv \sigma_1\sigma_2\sigma_1^{-1}\sigma_2^{-1}\in \pi_1(\R^2\setminus {P_1,P_2})$ (cf with Figure \ref{fig:Example}).  Then $(2,0)\in \mathrm{Ad}(\gamma)$ as we can erase $\sigma_1$, $\sigma_1^{-1}$ to obtain the null word:
\[
\sigma_1\sigma_2\cancel{\sigma_1^{-1}}\sigma_2^{-1}=\sigma_1\sigma_2\sigma_2^{-1}\equiv \sigma_1, \ \ \cancel{\sigma_1}=1.
\]
However the minimum is achieved by $(0,2)$ (which similarly belongs to $\mathrm{Ad}(\gamma)$) since $|U_1|>|U_2|$:
\[
\A(\gamma)=2|U_2|.
\]
\begin{figure}[b!]
\begin{center}
\includegraphics[scale=0.6]{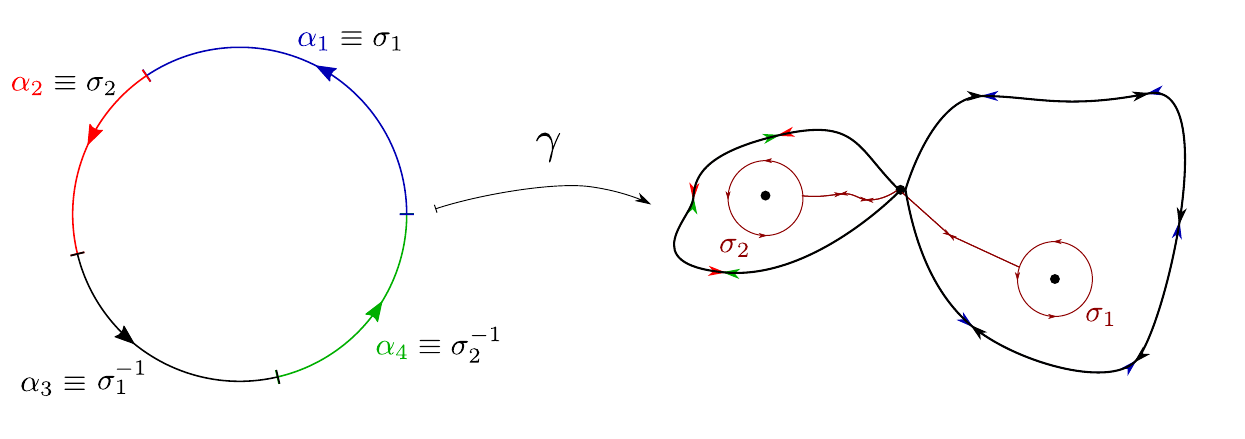}
\caption{A non-injective curve supported on a $\Gamma\subset \R^2\setminus\{P_1,P_2\}$ with a representing word $\gamma\equiv \sigma_1\sigma_2\sigma_1^{-1}\sigma_2^{-1}$.  Here both $(0,2), (2,0)\in\mathrm{Ad}(\gamma)$ but the optimum is given by $(0,2)$ since $|U_1|>|U_2|$.}\label{fig:Example}
\end{center}
\end{figure}
It is worth noting that $\mathcal A(\gamma)$ depends not only on the homotopy class of $\gamma$ in $F(n)$, but also on the areas $\mathcal L^2(U_i)$; Specifically,  two curves,  $\gamma$ and $\eta$ which are homotopically equivalent,   may have different functions $f^\gamma$ and $f^\eta$. \\ 
\smallskip

The proof of Theorem \ref{teo1_intro} is obtained by separately proving the lower bound (in Subsection \ref{sbsct:LB}) and the upper bound (in Subsection \ref{sbsct:UB}).   The lower bound is proven by observing that the set $\mathrm{Ad}(\gamma)$ contains information about some common features shared by any $u:D\rightarrow \R^2$ with $u\res \Su =\gamma$.  In particular,  we observe that if $z_i\in U_i$ are regular points for such a $u$,  then $\#(u^{-1}(z_i))\in \N$, and the $n$-tuple $(\#(u^{-1}(z_1)),\ldots ,\#(u^{-1}(z_n))) \in \mathrm{Ad}(\gamma)$.  Geometrically,  for a fixed family of $\{z_i\in U_i\}_{i=1}^n$,  this connection is intuitively justified by observing that the null word must be (homotopically) equivalent to $u\res \partial \tilde{D}$ where $ \tilde{D}\subset D$ is a small disk with $u^{-1}(z_i)\cap \tilde{D}= \varnothing$  for all $i=1,\ldots,N$.  If we retract the disk $\Su=\partial D$ onto $\partial \tilde D$ we must thus overcome all the points in $u^{-1}(z_i)$,  and each time we trepass a point in  in $u^{-1}(z_i)$,  this results in a cancellation of $\sigma_i$ or $\sigma_i^{-1}$ in the word describing $\gamma$.   This is the content of the crucial Lemma \ref{lem:Crucial} establishing this connection.  The area formula now allows to link the area of the $u$ with its multiplicities in $U_i$ and Lemma \ref{lem:Crucial} links the multiplicities of $u$ with $\gamma$.  \\

The upper bound is essentially constructive and is proven by first observing that the area functional is subadditive with respect to concatenation of curves.  Then we notice that a list $(k_1,\ldots,k_n)\in \mathrm{Ad}(\gamma)$ allows us to express $\gamma$ as a convenient concatenation of simple Jordan curves,  for which we can compute the area with classical results. \\

After completing our analysis o curves possessing property (P) we observe that Theorem \ref{teo1_intro} extends to the case of general curves with possibly countably many self-intersections.  In this case,  the bounded connected components $U_i$ of $\R^2\setminus\Gamma$ are countably many,  and so we can fix a sequence of points $P_i\in U_i$, $i\geq1$. For all $n\geq1$ we consider $\gamma^n\in F(n)$, the representative of $\gamma$ in $\pi_1(\R^2\setminus \{P_1,\dots,P_n\})$; The value of $\mathcal A(\gamma)$ is then obtained as a limit process.
\begin{theorem}\label{teo2_intro}
Let $\gamma:\Su\rightarrow \R^2$ be any Lipschitz curve such that  $\R^2\setminus\Gamma$ consists of countable many bounded connected components $U_i$, $i\geq1$.  Then, for all $n\geq1$ there is a function $f^n:F(n)\rightarrow \N^n$ such that 
\begin{align}
	\mathcal A(\gamma)=\lim_{n\rightarrow \infty}\sum_{i=1}^nf^n_i(\gamma)\mathcal L^2(U_i). 
\end{align}	
In particular the limit in the right-hand side does not depend on the choice of the indexes for $P_i$, i.e., if $P_{\tau(i)}$ is a reordering of the sequence $(P_i)_{i\geq 1}$, then there are functions $g^n:F(n)\rightarrow \N^n$ such that 
	$$\lim_{n\rightarrow \infty}\sum_{i=1}^nf^n_i(\gamma)\mathcal L^2(U_i)=\lim_{n\rightarrow \infty}\sum_{i=1}^ng^n_i(\gamma)\mathcal L^2(U_{\tau(i)}).$$
	\end{theorem}

Again,  as $f^n$ is explicit, also $\mathcal A(\gamma)$ is explicitly determined (see Theorem \ref{teo_main2} for details).  The proof of Theorem \ref{teo2_intro} relies in approximating the curve $\gamma$ with a smooth curve $\gamma_\varepsilon$ satisfying (P) and for which we can invoke Theorem \ref{teo1_intro}.  \\
\smallskip

\textbf{Organization of the paper.} In Section \ref{sct:Preliminaries} we set up the notation and the main ingredients to state the main results in their complete form.  We also introduce all the main tools that we will use in Section \ref{sct:PfMain},  where we prove the main Theorem   \ref{teo1_intro}.  Finally in Section \ref{sct:Cnn} we prove Theorem \ref{teo2_intro}  dealing with curves splitting $\R^2$ in possibly countably many bounded connected component.

\section{Preliminaries and main statement}\label{sct:Preliminaries}

We will denote by $D=\{x\in \R^2:|x|\leq1\}=\overline B_1(0)\subset\R^2$ the closed ball with radius $1$ and center the origin. Thus 
$$\Su:=\{x\in \R^2:|x|=1\}=\partial D.$$

\subsection{Area-spanning equivalent curves}
We here collect some result required to prove Theorem \ref{teo_main} and \ref{teo_main2}.
\begin{definition}
{\rm	Given the function $\psi:[-1,1]\rightarrow [0,1]$, $\psi(t)=1-|t|$, let $\Delta\subset\R^2$ be defined as
	$$\Delta:=\{(x_1,x_2)\in \R^2:-\psi(x_1)<x_2<\psi(x_1)\}.$$
	We distinguish
	$\Delta^\pm:=\Delta\cap \{x_2\gtrless0\}$. }
\end{definition}
Notice that $\Delta\subset D$, and that there is a bi-Lipschitz homeomorphism\footnote{This can be explicitely built, for instance setting $\Phi(0)=0$ and $\Phi(x_1,x_2):=\frac{|x_1|+|x_2|}{\sqrt{x_1^2+x_2^2}} (x_1,x_2)$. } \begin{align}\label{Phi_def}
\Phi:\overline\Delta\rightarrow D.
\end{align} In particular, composing $\Phi$ with any function $\phi:D\rightarrow \R^2$, we get a function $\phi\circ\Phi:\overline\Delta\rightarrow \R^2$ satisfying 
\begin{align*}
A(\phi\circ\Phi;\Delta):=&\int_{\Delta}|\det(\nabla (\phi\circ\Phi))|dx=\int_{\Delta}|\det(\nabla \phi (\Phi(x)))||\det(\nabla \Phi(x))|dx\\
 =&\int_{D}|\det(\nabla \phi (y))| dy =A(\phi).
\end{align*}
Hence we deduce
that 
\begin{align}\label{area_delta}
	\mathcal A(\gamma)=\inf\{A(\psi;\Delta) \ | \ \psi\in C^{0,1}(\overline \Delta;\R^2)\text{ such that } \psi\circ\Phi^{-1}=\gamma\text{ on }\Su\},
\end{align}
where
$$A(\psi;\Delta):=\int_\Delta|\det(\nabla \psi)|dx.$$
Similarly, using a bi-Lipschitz map $\Psi$ between $D$ and $\Delta^+$, it is easy to see that 
\begin{align}\label{area_delta+}
	\mathcal A(\gamma)=\inf\{A(\psi;\Delta^+)\ | \ \psi\in C^{0,1}(\overline \Delta^+;\R^2)\text{ such that } \psi\circ\Psi^{-1}=\gamma\text{ on }\Su\},
\end{align}
where 
$$A(\psi;\Delta^+):=\int_{\Delta^+}|\det(\nabla \psi)|dx.$$

Let $\theta_1,\theta_2\in  \Su $ be two distinct points. We denote by $\widehat{\theta_1\theta_2}$ the unique counterclockwisely oriented arc of $\Su$ with endpoints $\theta_1$ and $\theta_2$, starting from $\theta_1$ and ending at $\theta_2$. In particular, $\widehat{\theta_1\theta_2}\cup \widehat{\theta_2\theta_1}=\Su$. We also denote by $-\widehat{\theta_1\theta_2}$  the clockwisely oriented arc starting from $\theta_1$ and ending at $\theta_2$; hence $-\widehat{\theta_1\theta_2}$ coincides with $\widehat{\theta_2\theta_1}$, but has opposite orientation.
 
\begin{definition}
{\rm	An oriented arc $c\subset \Su$ is an arc of the form $\widehat{\theta_1\theta_2}$ or $-\widehat{\theta_1\theta_2}$.
	Given two oriented arcs $c$ and $d$ we call $\gamma:c\rightarrow d$ an \textit{arcs homeomorphism} if $\gamma$ is one to one, bi-Lipschitz, and preserves orientation (namely, it maps the starting point of $c$ in the starting point of $d$).}
\end{definition}

Let $\sigma:\Su\rightarrow \Su$; By identifying the domain $\Su$ with the interval $[0,2\pi)$ we can consider a lifting $\alpha$ of $\sigma$, namely $\alpha:[0,2\pi)\rightarrow \R$, $\sigma(t)=e^{i\alpha(t)}$. 
\begin{definition}{\rm
We say that $\sigma:\Su\rightarrow \Su$ is a {\it weakly monotonic re-parametrization} of $\Su$ if the lifting $\alpha$ is non-decreasing. }
\end{definition}

Given two curves $\gamma_1,\gamma_2\in  C^{0,1}(\Su;\R^2)$,
we denote by $d_F(\gamma_1,\gamma_2)$ the Frechet distance between them, defined as
\begin{align}
	d_F(\gamma_1,\gamma_2):=\inf_\sigma\big(\max\{|\gamma_1(t)-\gamma_2\circ\sigma(t)|, \; t\in \Su\}\big),
\end{align}
where the infimum is computed among all weakly monotonic re-parametrizations  $\sigma:\Su\rightarrow \Su$. Of course, if there is a $\sigma$ such that $\gamma_1=\gamma_2\circ\sigma$ then $d_F(\gamma_1,\gamma_2)=0$. The converse is not always true.   The following fact is well-known; we sketch the proof for the reader convenience.
%

\begin{prop}\label{prop:continuity}
	Let $\gamma_n:\Su\rightarrow \R^2$, $n\geq1$, be a sequence of Lipschitz curves satisfying 
\begin{equation}\label{eqn:Bound Length}
\ell(\gamma_n):=\int_{\Su}|\dot\gamma_n|dt\leq C,\qquad \forall n\geq1,
\end{equation}
	for some positive constant $C$. Assume there is a Lipschitz curve $\gamma:\Su\rightarrow\R^2$ such that 
	$$d_F(\gamma_n,\gamma)\rightarrow 0\qquad \text{ as }n\rightarrow\infty.$$
	Then $\displaystyle \lim_{n\rightarrow\infty}\mathcal A(\gamma_n)=\mathcal A(\gamma)$. 
\end{prop}
\begin{proof}
	Let us fix $n\in \mathbb N$. By definition of Frechet distance, there is a weakly monotonic reparametrization $\sigma:\Su\rightarrow \Su$ such that $\|\gamma_n\circ\sigma-\gamma\|_{L^\infty(\Su;\R^2)}\leq  d_F(\gamma_n,\gamma)+\frac1n$. Denote also by $\alpha:[0,2\pi)\rightarrow\R$ a continuous non-decreasing lifting of $\sigma$, namely $\sigma(t)=e^{i\alpha(t)}$. 
	
	Let $\varphi\in C^{0,1}(D;\R^2)$ be such that $\varphi\res\Su=\gamma$, and for $\delta\in(0,\frac12)$ fixed let  $\varphi_\delta(x):=\varphi\left(\frac{x}{1-2\delta}\right)$  be its rescaling on $B_{1-2\delta}(0)$. We define $\psi\in C^{0,1}(D;\R^2)$ as
	\begin{align}
		\psi(x)=\begin{cases}
			\varphi_\delta(x)&\text{if }|x|\leq 1-2\delta,\\
			\frac{1-\delta-|x|}{\delta}\gamma(\frac{x}{|x|})+\frac{|x|+2\delta-1}{\delta}(\gamma_n\circ\sigma)(\frac{x}{|x|})&\text{if }1-2\delta<|x|\leq 1-\delta\\
			\gamma_n\left(\exp\left(i\frac{1-|x|}{\delta}\alpha(t)+i\frac{|x|+\delta-1}{\delta}t\right)\right)&\text{if }1-\delta<|x|\leq 1,
		\end{cases}
	\end{align}  
where $t\in [0,2\pi)$ is introduced in such a way that  $\frac{x}{|x|}=e^{it}$.
	Passing in polar coordinates, straightforward computations shows that there is a constant $C>0$ independent of $n$ so that  
	
	\[
	|\partial_\rho\psi|\leq  C\delta^{-1}\|\gamma_n\circ\sigma-\gamma\|_{L^\infty}, \qquad \ \ |\partial_\theta\psi|\leq C(|\dot\gamma|+|\dot\gamma_n|)
	\] 
and so,  by  \eqref{eqn:Bound Length}
\[
\int_{B_{1-\delta}(0)\setminus B_{1-2\delta}(0)}|\det(\nabla \psi)|dx\leq  C\|\gamma_n\circ\sigma-\gamma\|_{L^\infty}\leq C\left(d_F(\gamma_n,\gamma)+\frac1n\right).
\]
Furthermore,  since $\psi(D\setminus B_{1-\delta})\subseteq  \gamma_n(\Su)$  it is straightforward to check that
$$\int_{D\setminus B_{1-\delta}(0)}|\det(\nabla \psi)|dx=0.$$
	Therefore, since $\psi\res\Su=\gamma_n$, we have
\[
\mathcal A(\gamma_n)\leq C\left(d_F(\gamma_n,\gamma)+\frac1n\right)+\int_{B_{1-\delta}(0)}|\det\nabla \varphi_\delta|dx=C\left(d_F(\gamma_n,\gamma)+\frac1n\right)+A( \varphi;D)
\]
where we exploited 
\[
\int_{B_{1-\delta}(0)}|\det\nabla \varphi_\delta|dx=\int_{B_{1}(0)}|\det\nabla \varphi|dx=A( \varphi;D).
\]
The arbitrariness of $\varphi$ implies that $\mathcal A(\gamma_n)\leq \mathcal A(\gamma)+o_n(1)$, where $o_n(1)\rightarrow 0$ as $n\rightarrow+\infty$. By switching the role of $\gamma$ and $\gamma_n$ we infer also $\mathcal A(\gamma)\leq \mathcal A(\gamma_n)+o'_n(1)$, where $o'_n(1)\rightarrow 0$ as $n\rightarrow+\infty$. Hence, passing to the limit as $n\rightarrow+\infty$ we get the thesis.
\end{proof}

From Proposition \ref{prop:continuity}  one readily infers that $\mathcal A(\gamma)$ is not sensible to negligible perturbation in Frechet distance.  In particular, we get the following Corollaries.
\begin{cor}\label{Distant0curvesHaveSameArea}
	Let $\gamma:  \Su \rightarrow\R^2$ be a Lipschitz curve. Then, for all Lipschitz curves  $\widehat \gamma: \Su\rightarrow\R^2$ with $d_F(\widehat{\gamma},\gamma)=0$ it holds $\mathcal A(\gamma)=\mathcal A(\widehat\gamma)$.
\end{cor}
\begin{proof}
	It is sufficient to choose $\gamma_n=\widehat\gamma$ in Proposition \ref{prop:continuity}, for all $n\geq0$.
\end{proof}

\begin{cor}\label{cor:area_reparametrization}
	Let $\gamma: \Su\rightarrow \R^2$ be a Lipschitz curve, and let $\sigma:\Su\rightarrow \Su$ be a bi-Lipschitz homeomorphism. Then $\mathcal A(\gamma)=\mathcal A(\gamma\circ \sigma)$.  
\end{cor}

\begin{definition}[Concatenation of curves] \label{def28}{\rm
Let $\gamma_1,\gamma_2\in  C^{0,1}(\Su;\R^2)$ be two curves and assume that there is $\theta\in \Su$ such that $\gamma_1(\theta)=\gamma_2(\theta)$. We define the concatenation $\gamma_1\star\gamma_2$ as follows: If  $\theta=e^{i\alpha}$, we set
$$\gamma_1\star\gamma_2(t):=\begin{cases}
\gamma_1(e^{i2(t-\alpha)})&\text{if }t\in[\alpha,\alpha+\pi)\\
\gamma_2(e^{i2(t-\pi-\alpha)})&\text{if }t\in[\alpha+\pi,\alpha+2\pi). 
\end{cases}$$
Similarly, if $\gamma_1,\gamma_2\in  C^{0,1}(\Su;\R^2)$ are curves such that there are $\theta_1=e^{i\alpha_1},\;\theta_2=e^{i\alpha_2}\in \Su$ such that $\gamma_1(\theta_1)=\gamma_2(\theta_2)$, we can define the concatenation $\gamma$ of them simply as
$$\gamma:=\gamma_1\star \widehat \gamma_2,$$
where $\widehat \gamma_2(t)=\gamma_2(e^{i(t-\alpha_1+\alpha_2)})$, so that $\widehat \gamma_2(\theta_1)=\gamma_1(\theta_1)$. In this case, we still denote the concatenation by $\gamma_1\star\gamma_2$, when there is no risk of confusion.}
\end{definition}

\begin{lemma}\label{lem_subadd}
	Let $\gamma_1,\gamma_2\in C^{0,1}(\Su;\R^2)$ curves such that  $\gamma_1(\theta_1)=\gamma_2(\theta_2)$, for some $\theta_1,\theta_2\in \Su$. Then $\gamma:=\gamma_1\star\gamma_2$ is a closed Lipschitz curve satisfying 
	$$\mathcal A(\gamma)\leq\mathcal A(\gamma_1)+\mathcal A(\gamma_2).$$
\end{lemma}
\begin{proof}
	Let $T$ denote the triangle in $\R^2$ with vertices the three points $(0,1)$, $(-2,-1)$, $(2,-1)$.
	The map $\Phi_{T,\Delta}:\overline T\rightarrow\overline\Delta$ defined as
	\begin{align}
		\Phi_{T,\Delta}(x,y)=\begin{cases}
			(x_1,x_2)&\text{if }x_2\geq0,\\
			\left(x_1\frac{1+x_2}{1-x_2},x_2\right)&\text{if }x_2<0,
		\end{cases}
	\end{align}
	is Lipschitz continuous and one-to-one between the interior of $T$ and $\Delta$. Further, the whole segment with vertices $(-2,-1)$ and $(2,-1)$ is mapped to $(-1,0)$. If $\Phi_{\Delta^+,T}:\Delta^+\rightarrow T$ is the homothety $(x,y)\mapsto  (2x,2y-1)$ then $\Phi_{\Delta^+,\Delta}:=\Phi_{T,\Delta}\circ\Phi_{\Delta^+,T}$ is a Lipschitz map sending $\Delta^+$ onto $\Delta$ and the segment $[-1,1]\times\{0\}$ to $(0,-1)$. Similarly we define $\Phi_{\Delta^-,\Delta}:\Delta^-\rightarrow\Delta$.
	
	We now assume (up composing $\gamma_i$ with $\Phi$ in \eqref{Phi_def}) that $\gamma_i:\partial \Delta\rightarrow \R^2$, $i=1,2$, and moreover, without loss of generality, we suppose that $\theta_1=(0,-1)$ and $\theta_2=(0,1)$. If now $\varphi_1$ and $\varphi_2$ are Lipschitz maps such that $\varphi_i\res\partial \Delta=\gamma_i$, we define
	\begin{align}
	\varphi(x):=\begin{cases}
				\varphi_1\circ\Phi_{\Delta^+,\Delta}(x)&\text{if }x\in \overline{\Delta^+},\\
				\varphi_2\circ\Phi_{\Delta^-,\Delta}(x)&\text{if }x\in \overline{\Delta^-},
	\end{cases}
	\end{align}
	and it turns out that $\varphi$ is Lipschitz continuous on $\overline{\Delta}$, and its boundary datum is $\gamma=\gamma_1\star\gamma_2$.
	
	Furthermore, we readily infer
	$$\mathcal A(\gamma)\leq A(\varphi,\Delta)=A(\varphi_1,\Delta)+A(\varphi_2,\Delta),$$
	so the thesis follows by arbitrariness of $\varphi_i$, $i=1,2$ and \eqref{area_delta}.
\end{proof}

\begin{definition}[Null curves]\label{def:null_curves}
	{\rm	Let $\gamma:\Su\rightarrow\R^2$ be a Lipschitz curve. We say that $\gamma$ is \textit{$1$-null} if there are two distinct points $\theta_1,\theta_2\in \Su$ and an arcs homeomorphism $\sigma:\widehat{\theta_1\theta_2}\rightarrow-\widehat{\theta_1\theta_2}$ such that 
		$$\gamma(t)=\gamma(\sigma(t))\qquad \forall t\in \widehat{\theta_1\theta_2}.$$
		We say that a Lipschitz curve $\gamma:\mathbb S^1\rightarrow\R^2$ is \textit{finitely null} if it is the concatenation of finitely many $1$-null curves. 
		Finally we say that a Lipschitz curve $\gamma:\Su\rightarrow\R^2$ is \textit{null} if it is the uniform limit of finitely null curves.}
\end{definition}

\begin{lemma}\label{lem_null0}
	Let $\gamma:\Su\rightarrow\R^2$ be a Lipschitz null curve. Then $\mathcal A(\gamma)=0$.
\end{lemma}
\begin{proof}
	\textit{Step 1:} Assume in this step that $\gamma$ is $1$-null. Hence  we can consider Lipschitz homeomorphisms $\sigma^\pm:[-1,1]\rightarrow \pm\widehat{\theta_1\theta_2}$ such that 
	$$\sigma^\pm(-1)=\theta_1,\qquad \sigma^\pm(1)=\theta_2,$$
	and such that 
	$$\gamma\circ\sigma^+(t)=\gamma\circ\sigma^-(t)\qquad\text{for all }t\in [-1,1].$$
	Indeed, once $\sigma^+$ has been defined, if $\sigma$ is the paramentrization as in Definition \ref{def:null_curves}, it is sufficient to take $\sigma^-:=\sigma\circ \sigma^+$.
	We define $\phi:\Delta\rightarrow\R^2$ as
	$$\phi(x_1,x_2)=\gamma\circ\sigma^+(x_1),$$
	which satisfies $\phi\res \partial \Delta=\gamma\circ \widetilde\sigma$, where $$
	\widetilde\sigma(x_1,x_2)=\begin{cases}
		\sigma^+(x_1)&\text{ for }(x_1,x_2)\in D_1^+\\
		\sigma^-(x_1)	&\text{ for }(x_1,x_2)\in D_1^-\\
		\theta_1&\text{ for }(x_1,x_2)=(-1,0)\\
		\theta_2&\text{ for }(x_1,x_2)=(1,0)
	\end{cases}$$
	is a bi-Lipschitz homeomorphism between $\partial \Delta$ and $\Su$.
	It follows that 
	\begin{align}
		D_{x_1} \phi(x_1,x_2)=(\gamma\circ\sigma^+)'(x_1),\qquad D_{x_2} \phi(x_1,x_2)=0,
	\end{align}
	so that $A(\phi;\Delta)=0$. As a  consequence, if $\Phi$ is the map in \eqref{Phi_def}, it follows that  
	since $\phi\circ \Phi^{-1}=\gamma\circ \widetilde\sigma\circ \Phi^{-1}$ on $\Su$, we have $d_F(\phi\circ \Phi^{-1}\res\Su, \gamma)=0$, and we infer $\mathcal A(\gamma)=0$ from \eqref{area_delta} and Corollary \ref{Distant0curvesHaveSameArea} .
%
%
%
%
%
%
%
	
	\textit{Step 2:} Let $\gamma_1:\Su\rightarrow\R^2$ be a Lipschitz curve such that $\mathcal A(\gamma_1)=0$, and let $\gamma_2$ be $1$-null and such that  there exists $\theta\in \Su$ with $\gamma_1(\theta)=\gamma_2(\theta)$. Then $\gamma:=\gamma_1\star\gamma_2$ satisfies $\mathcal A(\gamma)=0$ by Lemma \ref{lem_subadd}.
	This shows, by induction, that any finitely null curve $\gamma$ enjoys $\mathcal A(\gamma)=0$. 
	
	Finally, let $\gamma$ be a null curve and let $\gamma_n$ be finitely null curves tending uniformly to $\gamma$. Since uniform convergence implies $d_F(\gamma_n,\gamma)\rightarrow0$, the thesis follows from Proposition \ref{prop:continuity}. 
\end{proof}

\begin{prop}\label{prop:nullcurve}
	Let $\gamma\in C^{0,1}(\Su;\R^2)$ be a curve and let $\widehat \gamma\in C^{0,1}(\Su;\R^2)$ be a null curve such that $\gamma(\theta)=\widehat\gamma(\widehat\theta)$. Then 
	$$\mathcal A(\gamma\star\widehat\gamma)=\mathcal A(\gamma).$$
\end{prop}

\begin{proof}
\textit{Step 1:} Assume that $\widehat\gamma $ is $1$-null as in Definition \ref{def:null_curves}, and that moreover $\theta_1=\widehat\theta$.  Let $\widetilde\sigma:[-1,0]\rightarrow \widehat{\theta_1\theta_2}$ be a bi-Lipschitz homeomorphism, and let us extend it on $[-1,1]$ in order to have $\widetilde\sigma(t)=\widetilde\sigma(-t)$ for all $t\in[0,1]$. Then define the Lipschitz map $\varphi^-:\Delta^-\rightarrow \R^2$ as
$\varphi^-(x_1,x_2):= \widehat\gamma\circ\widetilde\sigma(  |x_1|+|x_2|)$; notice that $\varphi^-(x)= \widehat\gamma(\widehat\theta)=\gamma(\theta)$ for all $x\in \partial \Delta^-$ with $x_2<0$.  Thence
\begin{equation}
\nabla \varphi^-(x)=\left(\begin{array}{c   }
\frac{x_1}{|x_1|} (\widehat\gamma\circ\widetilde\sigma)'(  |x_1|+|x_2|)\\
\frac{x_2}{|x_2|} (\widehat\gamma\circ\widetilde\sigma)'(  |x_1|+|x_2|)
\end{array}\right)
\end{equation}
and thus,  being the two rows of the matrix $\nabla \varphi^-(x)$ linearly dependent,  $|\det(\nabla\varphi^-)|=0$. 

Let $\varphi^+:\overline\Delta^+\rightarrow \R^2$ be any Lipschitz map satisfying $\varphi^+(x_1,0)=\widehat\gamma\circ\widetilde\sigma(|x_1|)=\varphi^-(x_1,0)$ and $\varphi^+(x_1,x_2)=\gamma(s(x_1))$, where $s:[-1,1]\rightarrow \Su$ is any Lipschitz homomorphism satisfying $s(-1)=s(1)=\theta$. In this way the map
$$\varphi(x):=\begin{cases}
	\varphi^+(x)&\text{if }x\in \overline\Delta\cap\{x_2\geq0\}\\
	\varphi^-(x)&\text{if }x\in \overline\Delta\cap\{x_2<0\},
\end{cases}
$$
is Lipschitz from $\overline\Delta$ to $\R^2$ and $\varphi\res\partial \Delta$ is a reparametrization of $\gamma$.

Now, $\varphi^+$ has, as boundary datum $\varphi^+\res\partial\Delta^+$, a Lipschitz reparamentrization of $\gamma\star\widehat\gamma$. Thus
\begin{align}
	\mathcal A(\gamma)\leq A(\varphi;\Delta)=\int_{\Delta^+}|\det(\nabla \varphi^+)|dx=A(\varphi^+;\Delta^+).
\end{align}
So, by arbitrariness of $\varphi^+$ and from \eqref{area_delta+} we infer $\mathcal A(\gamma)\leq \mathcal A(\gamma\star\widehat\gamma)$.
The opposite inequality instead follows from Lemmas \ref{lem_subadd} and \ref{lem_null0}. The thesis is achieved in the case that $\widehat \gamma$ is $1$-null and such that $\theta_1=\widehat\theta$.

\textit{Step 2:} Let $\widehat \gamma$ be arbitrary. Then from Lemma \ref{lem_subadd} we get
$$\mathcal A(\gamma\star\widehat\gamma\star \widehat\gamma^{-1})\leq \mathcal A(\gamma\star\widehat\gamma)+\mathcal A(\widehat\gamma^{-1})= \mathcal A(\gamma\star\widehat\gamma),$$ 
where the equality follows from Lemma \ref{lem_null0}. Now, $\widehat\gamma\star \widehat\gamma^{-1}$ is $1$-null and satisfies the hypothesis of Step 1, so from the previous inequality we get $\mathcal A(\gamma)\leq  \mathcal A(\gamma\star\widehat\gamma)$. We then conclude the opposite inquality as in Step 1.
 \end{proof}

\subsection{Classes of curves} \label{sec:classes}
Let $\gamma:\partial D \rightarrow\R^2$ be a Lipschitz parametrization of a curve whose image we denote by 
\[
\Gamma:=\gamma(\Su).
\] 
We do not require that $\gamma$ is injective, so that $\Gamma$ might have self intersections. In this section we will make the assumption that $\R^2\setminus \Gamma$ is made by finitely many connected components $\{U_i\}_{i=1}^n$, namely Property (P) holds.\\

We denote by $U_0$ the unique unbounded connected component of $\R^2\setminus \Gamma$.  Moreover,  every $U_i$ with $i\geq 1$ is bounded and simply-connected.  Since $\Gamma$ is closed and $U_i$ is a connected component of $\R^2\setminus \Gamma$,  we infer that $\partial U_i\subseteq\Gamma$, for all $i$.  In particular $\mathcal H^1(\partial U_i)<\infty$. 
%

For all $i=1,\ldots,n$ we select $P_i\in U_i$ so that $\{P_1,\dots,P_n\}$ is a set of $n$ distinct points in $\R^2$. Further, for all $i=1,\ldots,n$, we choose closed balls $\overline B_r(P_i)\subset U_i$ in such a way that they are mutually disjoint. Let $P$ be a point in $\R^2\setminus (\cup_i\overline B_r(P_i))$.

\begin{definition}\label{def_puntibase}
{\rm	We say that a curve $\alpha:\Su\rightarrow \R^2\setminus \{P_1,\dots,P_n\}$ is based at $P$ if $\alpha((0,1))=P$.

Two Lipschitz curves $\alpha,\beta:\Su\rightarrow \R^2\setminus \{P_1,\dots,P_n\}$ based at $P$ are said to be homotopically equivalent if there is a Lipschitz homotopy $\Phi_{\alpha,\beta}:[0,1]\times \Su\rightarrow \R^2\setminus \{P_1,\dots,P_n\}$ such that $\Phi_{\alpha,\beta}(0,\cdot)=\alpha$, $\Phi_{\alpha,\beta}(1,\cdot)=\beta$, and $\Phi_{\alpha,\beta}(\cdot,(0,1))\equiv P$. In this case we write 
\begin{align}\label{homequi}
\alpha\equiv\beta.
\end{align}}
\end{definition}

We introduce the concept of winding number of a curve, which will be useful in the sequel.
\begin{definition}\label{def_homotopicEquivalence}
	Let $U$ be a bounded simply connected open set with Lipschitz boundary and let  $\gamma:\partial U\rightarrow\R^2$ be a Lipschitz curve, and $\Gamma:=\gamma(\partial U)$. For all $z\in \R^2\setminus \Gamma$ we introduce the winding number $\text{\rm link}(\gamma,z)\in \mathbb Z$ of $\gamma$ around $z=(z_1,z_2)$, defined by
	\begin{align}
		\text{\rm link}(\gamma,z):=\frac{1}{2\pi}\int_{\partial U}\left(\frac{\gamma_1(s)-z_1}{|\gamma(s)-z|^2}\dot\gamma_2(s)-\frac{\gamma_2(s)-z_2}{|\gamma(s)-z|^2}\dot\gamma_1(s)\right)ds.
	\end{align}
\end{definition}

Often, we will consider the winding number of curves defined on $\Su$, i.e. with $U=D$ in the previous definition.

\begin{definition}\label{def_sigma}
{\rm	For all $i=1,\dots,n$ we select a Lipschitz curve $\widetilde \sigma_i:\Su\rightarrow \R^2$ starting from $P$ and reaching a point in $\partial B_r(P_i)$; then we define $\sigma_i:=\widetilde \sigma_i\star \widehat \sigma_i\star \widetilde \sigma_i^{-1}$, where $\widehat \sigma_i$ is the costant speed parametrization of $\partial B_r(P_i)$ in counterclockwise order. The curve $\sigma_i:\Su\rightarrow\R^2$ will be a Lipschitz closed curve based at $P$ whose winding number $\textrm{link}(\sigma_i,P_j)$ around the point $P_j$ is 
	\[
	\mathrm{link}(\sigma_i,P_j)=\delta_{ij},\qquad\qquad \forall j\in \{1,\dots,n\}.
	\]
	Moreover, we assume that $\sigma_i((0,1))=P$ for all $i=1,\dots,n$ (we refer to Figures \ref{fig1:NotationAndExample} and \ref{fig2:ExamplefFiore} for a depiction of our notation and the complete scenario).  }
\end{definition}

The homotopy group of $\R^2\setminus \{P_1,\dots,P_n\}$ based at $P$ is the free group $F(n)$ on $n$ elements.  A basis for $\pi_1(\R^2\setminus \{P_1,\dots,P_n\})$ is given by the family $\{\sigma_1,\dots , \sigma_n\}$ (we recall that we always mean that the $\pi_1$ is meant to be based at $P$ even if is such dependence is not explicit).  An element $\eta$ of $F(n)$ is a string $\eta=(\eta_1,\dots,\eta_m)$ where every $\eta_i$ represents an element of the basis (or one of its inverse) and identifies the curve $\eta=\eta_1\star \ldots\star \eta_m$.   
To shortcut the notation, we will denote by 
\begin{align}\label{def_Sigma}
\Sigma(n):=\{ \sigma_1,\dots,
\sigma_n, \sigma_1^{-1},\dots, \sigma_n^{-1}\}.
\end{align}
We do not specify the dependence of $\Sigma$ on $\gamma$ since it will always be clear from the context.  Moreover, if $\eta\in \pi_1(\R^2\setminus \{P_1,\dots,P_n\})$ we will denote $(\eta_1,\dots,\eta_m)$ also by $\eta_1\dots\eta_m$, $\eta_i\in \Sigma(n)$.\\

Given a curve $\gamma\in C^{0,1}(\Su;\R^2\setminus \{P_1,\dots,P_n\})$ based at $P$, it belongs to a unique class in $\pi_1(\R^2\setminus \{P_1,\dots,P_n\})$.

\begin{definition}
Let $\gamma,\widehat \gamma\in \pi_1(\R^2\setminus \{P_1,\dots,P_n\})$ be curves based at $P$. By definition,  they belong to the same class in $\pi_1(\R^2\setminus \{P_1,\dots,P_n\})$ if and only if $\gamma\equiv \widehat \gamma$, where $\equiv$ is the relation \eqref{homequi} of Definition \ref{def_homotopicEquivalence}.
\end{definition}
Notice that if $\gamma\equiv \widehat \gamma$ and $\gamma_1\dots\gamma_m$ and $\widehat \gamma_1\dots\widehat\gamma_{\widehat m}$ are representatives of $\gamma$ and $\widehat \gamma $ in $F(n)$, then obviously $\gamma_1\dots\gamma_m=\widehat \gamma_1\dots\widehat\gamma_{\widehat m}$ in $F(n)$.

Given two Lipschitz curves $\alpha,\beta:\Su\rightarrow \R^2\setminus \{P_1,\dots,P_n\}$ based at $P$,  we can consider their concatenation $\alpha\star\beta$, according to Definition \ref{def28} where $\theta=P$.  
In such a case, denoting by $\alpha$ and $\beta$ themselves their representatives in $F(n)$, then the $\alpha\beta$ is a representative of $\alpha\star\beta$.


\begin{definition}[Generic representative]
{\rm Let $\alpha:\Su\rightarrow \R^2\setminus \{P_1,\dots,P_n\}$ be a Lipschitz curve based at $P$; if $\gamma_1\dots\gamma_{k}$ is a representative of $\alpha$, we call the representation $\gamma_1\dots\gamma_{k}$ of $\alpha$ a \textit{word} representing $\alpha$. The identity in $F(n)$ is called the {\it null word}. We also call a \textit{generic representative} of $\alpha$ any word of the form $\gamma_{1+a}\dots\gamma_{k+a}$  for all $a\in \mathbb Z$ (where $i+a$ is considered $\textrm{mod}(k)$), obtained from a representative $\gamma_1\dots\gamma_{k}$ of $\alpha$. The family of generic representative of $\alpha$ is noted as $$\jump{\alpha}:=\{\gamma=\gamma_1\dots\gamma_{k}:\gamma\text{ is a generic representative of }\alpha\}.$$}
\end{definition}
\begin{remark}\label{rem_coniugazione}
Notice that if $\gamma\in \jump{\alpha}$, we also have $\overline\gamma \gamma\overline\gamma^{-1}\in \jump{\alpha}$ for any $\overline\gamma\in \pi_1(\R^2\setminus \{P_1,\dots,P_n\})$; this follows easily by definition of generic representative. For this reason the notion of generic representative of $\gamma$ coincides with the conjugate class of $\gamma$ in $\pi_1(\R^2\setminus \{P_1,\dots,P_n\})$.
\end{remark}

Assume now $\alpha:\Su\rightarrow \R^2\setminus \{P_1,\dots,P_n\}$ is not based at $P$. Let $\widehat \beta:[0,1]\rightarrow  \R^2\setminus \{P_1,\dots,P_n\}$ be a Lipschitz curve so that $\widehat \beta(1)=\alpha((1,0))$, and $\widehat \beta(0)=P$. Up to reparamentrization on $\Su$, the curve $\widehat \alpha:=\widehat\beta^{-1}\star\widehat \beta$ is a $1$-null curve, and  $$ \widehat\beta\star\alpha\star\widehat \beta^{-1}$$ is based at $P$. We claim that the class $\jump{ \widehat\beta\star\alpha\star\widehat \beta^{-1}}$ does not depend on the choice of $\widehat \beta$. Indeed, if $\widetilde \beta$ is another Lipschitz curve such that $\widetilde \beta(1)=\alpha((1,0))$, $\widetilde \beta(0)=P$ then,  setting $\eta:=\widetilde\beta\star\widehat \beta^{-1}$,  we have
\[
\widetilde\beta\star\alpha\star\widetilde \beta^{-1} \equiv\widetilde\beta\star\widehat\alpha\star\alpha\star\widehat\alpha^{-1}\star\widetilde \beta^{-1}\equiv \eta\star \widehat\beta\star\alpha\star\widehat \beta^{-1}\star\eta^{-1}\equiv \widehat\beta\star\alpha\star\widehat \beta^{-1}  
\]
where we used Remark \ref{rem_coniugazione} to infer the last equivalence.  
Thanks to this fact the following definition is well-posed.

\begin{definition}\label{def_genclas}
	{\rm Let $\alpha:\Su\rightarrow \R^2\setminus \{P_1,\dots,P_n\}$ be a Lipschitz curve; we define the {\it generic class of $\alpha$} (or {\it conjugate class}),  denoted by $\jump{\alpha}$,  as the class of generic representatives  of $\beta \star\alpha\star \beta^{-1}$,  where $\beta:[0,1]\rightarrow  \R^2\setminus \{P_1,\dots,P_n\}$ is any $1$-null Lipschitz curve   such that $ \beta(1)=\alpha((1,0))$ and $ \beta(0)=P$.}
\end{definition}

\begin{remark}
	{Notice that, from the discussion preceeding Definition \ref{def_genclas}, the notion of generic class actually extends the notion of generic representative class to the case of curves not based at $P$. For this reason we have used the same symbol.  Furthermore, it is well-known that this class is invariant under free homotopy (i.e., without a base point). We recall this fact in the following theorem for the reader convenience.}
\end{remark}
\begin{lemma}\label{teo_2.18}
Let $\alpha,\beta:\Su\rightarrow \R^2\setminus \{P_1,\dots,P_n\}$ be Lipschitz curves and assume there is a Lipschitz homotopy $\Phi_{\alpha,\beta}:[0,1]\times \Su\rightarrow \R^2\setminus \{P_1,\dots,P_n\}$ such that $\Phi_{\alpha,\beta}(0,\cdot)=\alpha$ and $\Phi_{\alpha,\beta}(1,\cdot)=\beta$. 
Then $\alpha$ and $\beta$ have the same conjugate class, i.e. 
$$\jump{\alpha}=\jump{\beta}.$$
\end{lemma}
\begin{proof}
	Let
$\widehat \alpha:=\widehat\eta^{-1}\star\widehat \eta:[0,1]\rightarrow \R^2\setminus \{P_1,\dots,P_n\}$ be a Lipschitz $1$-null curve with $\widehat \eta:[0,1]\rightarrow  \R^2\setminus \{P_1,\dots,P_n\}$  such that $\widehat \eta(1)=\alpha((1,0))$ and $\widehat \eta(0)=P$.
We now define
$\widetilde \Phi:[0,1]\times [0,1]\rightarrow \R^2$ as 
$$\widetilde \Phi(t,s)=\Phi_{\alpha,\beta}\big((t-s)\vee0,(1,0)\big),$$
in such a way that $\widetilde \Phi(0,s)=\Phi_{\alpha,\beta}\big(0,(1,0)\big)=\alpha((1,0))$  and $\widetilde \Phi(1,s)=\Phi_{\alpha,\beta}\big(1-s,(1,0)\big)$ for all $s\in[0,1]$. We finally define $\Phi_{\alpha,\beta}^P:[0,1]\times[0,2\pi]\rightarrow \R^2\setminus \{P_1,\dots,P_n\}$ as
$$\Phi_{\alpha,\beta}^P(t,\theta):=\begin{cases}
	\Phi_{\alpha,\beta}(t,e^{i4\theta})&\text{if }\theta\in[0,\frac{\pi}{2}]\\
	\widetilde\Phi(t,\frac{2}{\pi}(\theta-\frac{\pi}{2}))&\text{if }\theta\in[\frac{\pi}{2},\pi]\\
	\widehat \alpha(\frac{2}{\pi}(\theta-\pi))&\text{if }\theta\in[\pi,\frac{3\pi}{2}]\\
	\widetilde\Phi(t,1-\frac{2}{\pi}(\theta-\frac{3\pi}{2}))&\text{if }\theta\in[\frac{3\pi}{2},2\pi],\\
\end{cases}$$
which turns is a Lipschitz homotopy\footnote{Where we have identified $[0,2\pi)$ with $\Su$.} between $\Phi_{\alpha,\beta}^P(0,\theta)=\widehat \eta^{-1} \star\alpha\star \widehat \eta $ and $\Phi_{\alpha,\beta}^P(1,\theta)=\widehat \gamma\star\beta$ (cf with Figure \ref{fig:Prop}); Here $\widehat \gamma:=\Phi_{\alpha,\beta}\big(1-\cdot,(1,0)\big)\star\widehat \alpha\star\Phi_{\alpha,\beta}^{-1}\big(1-\cdot,(1,0)\big)$ being a Lipschitz $1$-null curve connecting $P$ to $\beta((0,1))$. Hence the thesis follows from Definition \ref{def_genclas}.
\end{proof}

\begin{figure}[t!]
\begin{center}
\includegraphics[scale=0.7]{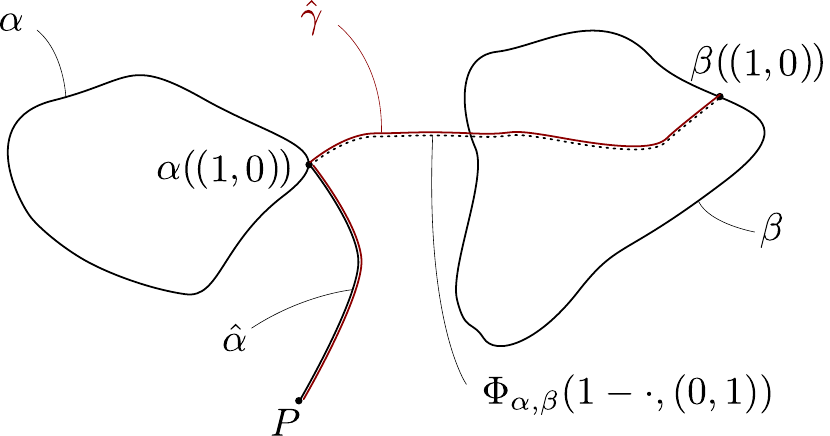}
\caption{A depiction of the situation in the proof of Lemma \ref{teo_2.18}. The main issue is to link $\alpha$,  $\beta$ to the same base point $P$.  To do that, the homotopy $\Phi_{\alpha,\beta}$ is exploited to produce the dotted curve $\Phi_{\alpha,\beta}\big(1-\cdot,(1,0)\big)$, which is then concatenated to $\hat{\alpha}$ to obtain $\hat{\gamma}$ (in red).}\label{fig:Prop}
\end{center}
\end{figure}

\begin{lemma}\label{teo_Ageneric}
Let $\gamma\in C^{0,1}(\Su;\R^2\setminus \{P_1,\dots,P_n\})$ and let $\beta\in C^{0,1}([0,1];\R^2\setminus \{P_1,\dots,P_n\})$ be a curve such that $\beta(0)=P$ and $\beta(1)\in \gamma(\Su)$. Then  $\mathcal A(\gamma)=\mathcal A(\beta\star\gamma\star\beta^{-1})$.
\end{lemma}
\begin{proof}
Notice that, using a  re-parametrization\footnote{This can be simply a rotation around $0$.} $\sigma:\Su\rightarrow \Su$, we see that $$\mathcal A(\beta\star\gamma\star\beta^{-1})=\mathcal A(\gamma\star\beta^{-1}\star\beta).$$
Hence, since $\beta^{-1}\star\beta$ is a $1$-null curve, the thesis readily follows from Proposition \ref{prop:nullcurve}.
\end{proof}

\subsection{Degree, multiplicity and area formula}

Let $U\subset\R^2$ be a bounded open set with Lipschitz boundary.
Given $u\in C^{0,1}(U;\R^2)$, we denote by $\mathcal R_u\subseteq\Om$ the set of regular points of $u$, namely the set of Lebesgue points $x$ of $\nabla u$. 
\begin{definition}
	Let $u\in C^{0,1}(U;\R^2)$ and let $A\subseteq U$ be a measurable set. For all $y\in \R^2$ the degree of $u$ on $A$ at $y$ is defined as
	$$\text{\rm deg}(u,A,y):=\sum_{x\in u^{-1}(y)\cap \mathcal R_u\cap A}\text{\rm sign}(\det\nabla u(x)),$$
	whenever the sum on the right-hand side exists.
\end{definition}

It is well-known (see, e.g., \cite{GMS}) that if $\det\nabla u\in L^1(U)$ then for a.e. $y\in \R^2$ the sum on the right-hand side is finite, $\text{\rm deg}(u,U,y)$ belongs to $ L^1(\R^2)$ and satisfies
\begin{align}\label{int_det_deg}
	\int_A\det\nabla udx=\int_{\R^2}\text{\rm deg}(u,A,y)dy,
\end{align}
for all measurable sets $A\subseteq U$. More generally, for all measurable $\varphi:\R^2\rightarrow \R$ one has
\begin{align}\label{int_det_deg2}
	\int_A\varphi(u(x))\det\nabla u(x)dx=\int_{\R^2}\varphi(y)\text{\rm deg}(u,A,y)dy.
\end{align}
Further, if $u$ is regular enough some other important properties of the degree are satisfied. Specifically, we summarize some of them:
\begin{itemize}
	\item[(D1)] If $u\in  C^{0,1}(\overline U;\R^2)$, then $\text{\rm deg}(u,U,\cdot)$ is locally constant on  $\R^2\setminus u(\partial U)$. Moreover, if $U_0\subset(\R^2\setminus u(\partial U))$ is the unbounded connected component of $\R^2\setminus u(\partial U)$ , then $\text{\rm deg}(u,U,\cdot)=0$ on $U_0$;
	\item[(D2)] If $u,v\in  C^{0,1}(\overline U;\R^2)$ and $u=v$ on $\partial U$, then 
	\begin{align}\label{equ_degree}
		\text{\rm deg}(u,U,y)=\text{\rm deg}(v,U,y),
	\end{align} 
	for a.e. $y\in \R^2$;
	\item[(D3)] If $u,v\in  C^{0,1}(\overline U;\R^2)$, and $\Phi:U\times[0,1]\rightarrow\R^2$ is an homotopy such that $\Phi\in C^{0,1}(\overline U\times[0,1];\R^2)$ with $\Phi(\cdot,0)=u(\cdot)$, $\Phi(\cdot,1)=v(\cdot)$, then  \eqref{equ_degree} holds for a.e. $y\in \R^2\setminus \Phi(\partial U\times[0,1])$.
\end{itemize}

We introduce the notion of degree for $\Su$-valued maps:
\begin{definition}
	{\rm	Let $w\in C^{0,1}(\partial U;\Su)$ where $U$ is a simply-connected bounded open set with $\partial U$ Lipschitz. Then we define the degree of $u$ on $\partial U$ as 
		\begin{align}
			\text{\rm deg}_{\partial U}(w):=\frac{1}{2\pi}\int_{\partial U}\left(w_1\frac{\partial w_2}{\partial s}-w_2\frac{\partial w_1}{\partial s}\right)ds=\frac{1}{2\pi}\int_{\partial U}\left(-w_2,w_1\right)\cdot\frac{\partial w}{\partial s}ds.
	\end{align} 
	where $\frac{\partial w}{\partial s}$ is the derivative computed tangentially  to $\partial U$.}
\end{definition}

\begin{remark}\label{rem_dominio}
	Let $U\subset \R^2$ be a simply connected and bounded open set with Lipschitz boundary, and let $\sigma:\Su\rightarrow \partial U$ be a Lipschitz parametrization of $\partial U$. 
	Let also $w\in C^{0,1}(\partial U;\Su)$, so that $v:=w\circ \sigma\in C^{0,1}(\Su;\Su)$; hence, by the change of variable formula one has
	\begin{align}
		\text{\rm deg}_{\Su}(v)&=\frac{1}{2\pi}\int_{\Su}\left(v_1\frac{\partial v_2}{\partial s}-v_2\frac{\partial v_1}{\partial s}\right)ds\nonumber\\
		&=\frac{1}{2\pi}\int_{\Su}\big(-w_2(\sigma(s)),w_1(\sigma(s))\big)\cdot \left(\frac{\partial w}{\partial \sigma}(\sigma(s))\frac{\partial \sigma}{\partial s}(s)\right)ds\nonumber\\
		&=\frac{1}{2\pi}\int_{\partial U}\big(-w_2,w_1\big)\cdot \frac{\partial w}{\partial \sigma}d\sigma=	\text{\rm deg}_{\partial U}(w).
	\end{align}
\end{remark}
An easy computation shows the following fact:

\begin{lemma}\label{lem_210}
	Let $\gamma:\Su\rightarrow \R^2$ be a Lipschitz curve and $z\in \R^2\setminus \Gamma$ (where $\Gamma=\gamma(\Su)$), then 
	\begin{align}\label{deg=wind}
		\text{\rm link}(\gamma,z)=\text{\rm deg}_{\Su}\left(\frac{\gamma-z}{|\gamma-z|}\right).
	\end{align}
\end{lemma}
\begin{remark}\label{rem_dominio2}
	As a consequence of Remark \ref{rem_dominio}, the previous lemma implies the following fact: 
	If $U$ is a Lipschitz bounded and simply connected open set and $\gamma:\partial U\rightarrow \R^2$, we choose a  Lipschitz parametrization $\sigma:\Su\rightarrow \partial U$, and infer
	\begin{align*}
		\text{\rm link}(\gamma\circ \sigma,z)=	\text{\rm deg}_{\Su}\left(\frac{\gamma\circ \sigma-z}{|\gamma\circ \sigma-z|}\right)=	\text{\rm deg}_{\partial U}\left(\frac{\gamma-z}{|\gamma-z|}\right)=	\text{\rm link}(\gamma,z).
	\end{align*} 
\end{remark}

Let now $\gamma:\Su\rightarrow\R^2$ be a Lipschitz curve, and let $A$ be a bounded connected component of $\R^2\setminus \Gamma$; let also $P\in A$ be arbitrary.
Assume that $u:D\rightarrow \R^2$ is a Lipschitz map with $u=\gamma$ on $\Su$; since the map $\gamma(\cdot)-P:\Su\rightarrow \R^2$ never vanishes, there is a Lipschitz homotopy $\Phi:\Su\times [0,1]\rightarrow \R^2$ such that
$$\Phi(\cdot,0)=\gamma(\cdot)-P,\qquad\Phi(\cdot,1)=\frac{\gamma(\cdot)-P}{|\gamma(\cdot)-P|}\qquad 0\notin\Phi(\Su\times[0,1]).$$
By extending $\Phi$ to $\overline D\times [0,1]$ in a Lipschitz way and with $\overline \Phi(\cdot,0)=u(\cdot)-P$, we see that the map $v:=\overline \Phi(\cdot,1)$, whatever the extension is, satisfies  
\begin{align}\label{eq_degrees}
	\text{\rm deg}(u(\cdot)-P,D,y)=\text{\rm deg}(v(\cdot),D,y)\qquad \text{for a.e. }y\in \R^2\setminus \Phi(\Su\times[0,1]),
\end{align}
by (D3). Furthermore, since $v(\Su)\subseteq \Su$, $\text{\rm deg}(v(\cdot),D,y)\equiv m$ is constant on $D$ (for some $m\in \mathbb Z$); we infer that there is a neighborhood $N$ of $0$ where \eqref{eq_degrees} holds for a.e. $y\in N$, and so, as $\text{\rm deg}(u(\cdot)-P,D,y)=\text{\rm deg}(u(\cdot),D,y+P)$, we see that $\text{\rm deg}(u,D,\cdot)=m$ in the neighborhood $N+P$ of $P$. In particular,  $\text{\rm deg}(u,D,\cdot)=m$ a.e. in $A$ (the connected component of $\R^2\setminus \Gamma$ containing $P$).
Finally, from Lemma \ref{lem_210}, we infer  
\begin{align*}
	\pi\text{\rm link}(\gamma,P)=&\pi\text{\rm deg}_{\Su}(v) 	 =  \int_D\det(\nabla v)dx=\int_{D}\text{\rm deg}(v,D,y)dy=m|D|=m\pi,
\end{align*}
where in the second equality we have used Stokes Theorem,  and in the third one equation \eqref{int_det_deg}. We conclude that 
\begin{align}\label{link=deg}
	\text{\rm link}(\gamma,y)=m=\text{\rm deg}(u,D,y),\qquad\qquad\text{for a.e. }y\in A.
\end{align}

The previous argument leads one to the following:
\begin{lemma}\label{lemma_linkdeg_gen}
	Let $U\subset\R^2$ be a simply-connected and bounded open set with Lipschitz boundary, let $\gamma:\partial U\rightarrow\R^2$ be a Lipschitz curve, and let $u:U\rightarrow \R^2$ be a Lipschitz map such that $u=\gamma$ on $\partial U$. Then 
	\begin{align}
		\text{\rm link}(\gamma,z)=\text{\rm deg}(u,U,z),\qquad\qquad\text{for a.e. }z\in \R^2.
	\end{align}
\end{lemma}
\begin{proof}
	Let $\sigma:\Su\rightarrow \partial U$ be a Lipschitz parametrization preserving orientation. From Remark \ref{rem_dominio2}
	\begin{align}
		\text{\rm link}(\gamma,z)=	\text{\rm link}(\gamma\circ \sigma,z)\qquad\text{ for a.e. }z\in \R^2.
	\end{align}
	Let $\phi:D\rightarrow \R^2$ be any Lipschitz extension of $\phi\res\Su=\sigma:\Su\rightarrow \partial U$. Hence, from \eqref{link=deg} we infer
	\begin{align}\label{224}
		\text{\rm link}(\gamma,z)=\text{\rm link}(u\circ\phi\res\Su,z)=\text{\rm deg}(u\circ \phi,D,z),\qquad\qquad\text{for a.e. }z\in \R^2.
	\end{align}
	Now, again by \eqref{link=deg} applied to $\sigma$ and $\phi$, one easily sees that 
	$$\text{\rm deg}(\phi,D,y)=\text{\rm link}(\sigma,y)=\chi_U(y)\qquad\qquad \text{for a.e. }y\in \R^2.$$
	Therefore, for all measurable maps $\varphi:\R^2\rightarrow \R$ we have, from \eqref{int_det_deg2},
	\begin{align*}
		\int_{\R^2}&\varphi(z)\text{\rm deg}(u\circ \phi,D,z)dz=
		\int_D\varphi(u(\phi(x)))\det(\nabla (u\circ\phi)(x))dx\\
		&=\int_D\varphi(u(\phi(x)))\det(\nabla u(\phi(x)))\det(\nabla \phi(x))dx=\int_{\R^2}\varphi(u(y))\det(\nabla u(y))\text{\rm deg}(\phi,D,y)dy\\
		&=\int_{U}\varphi(u(y))\det(\nabla u(y))dy=\int_{\R^2}\varphi(z)\text{\rm deg}(u,U,z)dz,
	\end{align*}
	from which we conclude $\text{\rm deg}(u\circ \phi,D,z)=\text{\rm deg}(u,U,z)$ for a.e. $z\in\R^2$. The thesis then follows from \eqref{224}.
\end{proof}

%
%
%
%

For a Lipschitz function $u:U\rightarrow \R^2$ and $z\in \R^2$ we denote by
\begin{align}\label{multiplicity}
	\mathrm{mul}(u,U,z):=\#(\{x\in U \ | \ u(x)=z\} ),
\end{align}
the multiplicity (by $u$) of $z$ on $U$.
From the area formula \cite[Theorem 8.9]{maggi2012sets} we have,  for all Lipschitz maps $u:U\rightarrow\R^2$,  that
\begin{equation}\label{eqn:area}
	\int_U |\det(\nabla u)|\d x = \int_{\R^2}\mathrm{mul}(u,U,z) \d z.
\end{equation}
We recall that, given $u:U\rightarrow\R^2$ Lipschitz, we denote by $\mathcal R_u\subset U$ the set of regular points for $u$. It is well-known that  $U\setminus \mathcal R_u$ is a negligible set with respect to the Lebesgue measure \cite{GMS}, and noting by \begin{align}\label{nu}
	N_u:=u(U\setminus \mathcal R_u),
\end{align} also $\mathcal L^2(N_u)=0$, since $u$ is Lipschitz. We also denote by $F_u\subset \R^2$ the set defined as
\begin{align}\label{fu}
	F_u:=\{z\in \R^2:u^{-1}(z) \text{ is a set with finite cardinality}\}.
\end{align}
Again, thanks to area formula \eqref{eqn:area}, we get that $\R^2\setminus F_u$ is $\mathcal L^2$ neglibible. We finally denote 
\begin{align}\label{eu}
	E_u:=F_u\setminus N_u.
\end{align}

\subsection{Main result}\label{sbsct:MainResult}

In order to introduce and prove our first main result we define the following operation that we call \textit{injection}.  Given an element of $\sigma_i\in \Sigma(n)$ in the  basis of the free group $F(n)$, it will be called an \textit{i-monoid} any object of the type $\beta  \sigma_i \beta^{-1}$ for $\beta\in F(n)$.  In other terms, an $i$-monoid is a conjugated of either $\sigma_i$ or $\sigma_i^{-1}$.

\begin{definition}\label{def_inj}
{\rm	Let $\gamma\in F(n)$  be a word.  Fix $(k_1,\ldots,k_n) \in \N^n$ an $n$-tuple of natural numbers.  We say that the curve $\gamma'\in F(n)$ is a \textit{$(k_1,\ldots,k_n)$-injection in $\gamma$} if the word  $\gamma'$ can be obtained by inserting $k_i$ times an $i$-monoid  into the word  $\gamma$ (for all $i=1,\ldots,n$).}
\end{definition}

For instance,  given $\gamma=\sigma_1\sigma_2\in F(3)$ then $\gamma'=\sigma_1\sigma_3\sigma_2$ is a $(0,0,1)$-injection in $\gamma$.  The curve $\gamma''=\sigma_1$ instead is a $(0,1,0)$-injection of $\sigma_2$ in $\gamma$ (we inserted $\sigma_2^{-1}$ into the word of $\gamma$).  But it is also a $(0,3,0)$-injection if we suitably insert $\sigma_2, \sigma_2^{-1},\sigma_2^{-1}$ in the word representing $\gamma$.  So there is not a unique way for seeing a curve $\gamma'$ as an injection in $\gamma$.  The null word is a $(1,1,0)$-injection of $\gamma$ (as well as a $(1+2h,1+2h,2h)$-injection for all $h\in \N$).

Morerover, the curve $\sigma_1\sigma_3\sigma_2^{-1}$ is a $(0,0,1)$-injection in $\gamma$, as it can be obtained by inserting the monoid $\sigma_2^{-1}\sigma_3\sigma_2$ at the end of the word representing $\gamma$; finally $\sigma_1\sigma_2\sigma_1^{-1}\sigma_2^{-1}\in F(2)$ is a $(1,0)$-injection in $\sigma_1$, as it is obtained from $\sigma_1$ by inserting the $1$-monoid $\sigma_2\sigma_1^{-1}\sigma_2^{-1}$.

\begin{remark}\label{rmk:sumofinj}
If $\gamma''$ is a $(k_1,\ldots,k_n)$-injection of $\gamma'$ and $\gamma'$ is a $(p_1,\ldots,p_n)$-injection of $\gamma$ then it is immediate that $\gamma''$ is also a $(k_1+p_1,\ldots,k_n+p_n)$-injection of $\gamma$.
\end{remark}

\begin{remark}\label{rem_attheend}
We also remark that the notion of injection in Definition \ref{def_inj} is equivalent to the following:  the curve $\gamma'$ is a \textit{$(k_1,\ldots,k_n)$-injection in $\gamma$} if the word representing $\gamma'$ can be obtained by inserting $k_i$ times a $i$-monoid \textit{at the end} of  the word representing $\gamma$ (for all $i=1,\ldots,n$). Indeed, if $\gamma'$ is obtained from $\gamma=\gamma_1\dots\gamma_k$ by inserting in it the $i$-monoid $\beta\sigma_i\beta^{-1}$, then it holds
$$\gamma'=\gamma_1\dots\gamma_h\beta\sigma_i\beta^{-1}\gamma_{h+1}\dots\gamma_k,$$
for some $h\in\{1,\dots,k\}$. On the other hand, by setting $\eta= (\beta^{-1}\gamma_{h+1}\dots\gamma_k)^{-1}$, we see that $\gamma'$ can also be obtained from $\gamma$ as
$$\gamma'=\gamma\eta\sigma_i\eta^{-1}.$$
\end{remark}

We can now define the main object of our concerns. Let $\gamma:\Su\rightarrow \R^2$ be a fixed curve satisfying Property (P). We recall that 
$\Gamma:=\gamma(\Su),$
and that the points $P_i$ are chosen in any bounded connected components $U_i$ of $\R^2\setminus \Gamma$ (which are $n$ components, $n\in \mathbb N$).

\begin{remark}\label{rem_deppoints}
Notice that the group $\pi_1(\R^2\setminus \{P_1,\dots,P_n\})$ and the equivalence class of $\gamma$ in it do not depend on the specific choice of the $P_i\in U_i$.
\end{remark}

 We define the following family of natural numbers.
\begin{equation}
\mathrm{Ad}(\gamma):=\left\{(k_1,\ldots,k_n)\in \N^n \ | \ \textit{the null word is a $(k_1,\ldots,k_n)$-injection of $\gamma$} \right\}.
\end{equation}
From Remark \ref{rem_deppoints} the set $\mathrm{Ad}(\gamma)$ depends only on the equivalence class of $\gamma$ in $\pi_1(\R^2\setminus \{P_1,\dots,P_n\})$ and hence not on the choice of the points $P_i$'s.  

\begin{remark}\label{rmk:NoDepOnGen} Notice also that,  if we change the set of generator of $\pi_1(\R^2\setminus \{P_1,\dots,P_n\})$ from $\sigma_i$ to some $\bar{\sigma}_i\equiv\sigma_i$ in $\pi_1$ we do now alter the set $\mathrm{Ad}(\gamma)$. 
\end{remark}

Remark \ref{rem_deppoints} and \ref{rmk:NoDepOnGen} ensures that the set is well defined and depends only on intrinsic properties of the curve $\gamma$ and not on its representation.  That being established our main Theorem is now the following. 
\begin{theorem}\label{teo_main}
Let $\gamma:\Su \rightarrow \R^2$ be a Lipschitz curve satisfying hypothesis (P).  Then
\begin{align}\label{eq_main}
\mathcal{A}(\gamma)=\min\left\{\left.\sum_{i=1}^n k_i |U_i|  \ \ \right| \ \ (k_1,\ldots,k_n)\in \mathrm{Ad}(\gamma) \right\}
\end{align}
where $ |U_i| $ is the standard Lebesgue measure of the connected component $U_i$ of $\R^2\setminus \gamma(\Su)$.
\end{theorem}

\subsection{Multiplicity and winding numbers}
In this section we describe how the degree theory relates with the classes of curves introduced in Section \ref{sec:classes}. This is necessary in order to prepare to the proof of Theorem \ref{teo_main}. 

We recall that $\gamma:\Su\rightarrow \R^2$ is a fixed curve satisfying (P) and that points $P_i$ have been selected according to Definition \ref{def_puntibase}.

\begin{lemma}\label{lem233}
	Let $U\subset\R^2$ be a bounded and simply connected domain  and let $\varphi:\Su\rightarrow \partial U$ a Lipschitz parametrization of $\partial U$. Let $u:U\rightarrow \R^2$ be a Lipschitz map, and assume that  $u\circ \varphi:\Su\rightarrow \R^2$ is a curve in $\R^2\setminus \{P_1,\dots,P_n\}$ with
	\begin{align}
		\mathrm{mul}(u,U,P_j)=0\qquad \qquad \forall j=1,\dots,n.
	\end{align}
	Then  $\jump{u\circ \varphi}=\jump{1}$. 
\end{lemma}
\begin{proof}
	By hypothesis on $U$ we can build a Lipschitz homothopy $\Phi:[0,1]\times \Su\rightarrow \overline U$ such that $\Phi(0,\cdot)=\varphi(\cdot)$  and $\phi(1,\cdot)$ is a constant point in $U$. The thesis then follows by applying Lemma \ref{teo_2.18} with $\alpha=u\circ \Phi(0,\cdot)$ and $\beta=u\circ\Phi(1,\cdot)$, the last loop belonging to $\jump{1}$.
\end{proof}

\begin{lemma}\label{lem_2.33}
	Let $U\subset\R^2$ and $\varphi$ be as in Lemma \ref{lem233}. Let $u:U\rightarrow \R^2$ be a Lipschitz map such that for all $j=1,\dots,n$, it holds $P_j\in E_u$, and assume that  $u\circ \varphi:\Su\rightarrow \R^2$ is a curve in $\R^2\setminus \{P_1,\dots,P_n\}$ satisfying 
	\begin{align}\label{mult_delta}
		\mathrm{mul}(u,U,P_j)=\delta_{ij}\qquad \qquad \forall j=1,\dots,n,
	\end{align}
for some $i\in\{1,\dots,n\}$.
	Then either $u\circ\varphi\in \jump{\sigma_i}$ or $u\circ\varphi\in \jump{\sigma_i^{-1}}$. 
\end{lemma}
Here, we recall that  $\sigma_i$'s are the loops of the basis $\Sigma(n)$ defined in \eqref{def_Sigma}. 
\begin{proof}
	From \eqref{mult_delta} it follows that 
	$$|\mathrm{deg}(u,U,P_j)|=\delta_{ij},$$
	and hence, from Lemma \ref{lemma_linkdeg_gen} it follows that 
	\begin{align}\label{link_hat}
		\left|\text{{\rm deg}}_{\partial U}\left(\frac{u-P_j}{|u-P_j|}\right)\right|=|\text{\rm link}(u\res \partial U,P_j)|=\delta_{ij}\qquad \qquad \forall j=1,\dots,n.
	\end{align}
	By hypothesis, we can built a Lipschitz homotopy $\phi:[0,1]\times \Su\rightarrow \overline U$ such that $\phi(0,\cdot)=\varphi(\cdot)$, and $\phi(1,\cdot)\equiv u^{-1}(P_i)$. Further, we can design $\phi$ in such a way that $  P_i) \notin\phi([0,1),\Su)$; in particular,  $u\circ\phi$ is an homotopy between $u\circ\varphi$ and the constant $P_i$ and for $\overline t<1$ large enough $u\circ \phi(\overline t,\cdot)$ is a Lipschitz curve in $B_r(P_i)$ never passing by $P_i$. We can then normalize it, and finally find a Lipschitz homotopy between $u\res\partial U$ and a curve $\beta:\Su\rightarrow \partial B_r(P_i)$. Now, from \eqref{link_hat} we infer that $\text{{\rm deg}}_{\partial B_r(P_i)}(\beta)=\pm 1$, and then there exists a homotopy between $\beta$ and either $\widehat \sigma_i$ or $\widehat \sigma_i^{-1}$ (where $\widehat \sigma_i$ is the constant speed counterclockwise parametrization of $\partial B_r(P_i)$ as in Definition \ref{def_sigma}). 
	
	Eventually, combining the previous homotopies, we can build a homotopy between $u\circ \varphi$ and either $\widehat \sigma_i$ or $\widehat \sigma_i^{-1}$; this implies, thanks to Lemma \ref{teo_2.18}, that either  $\jump{u\circ\varphi}=\jump{\sigma_i}$ or $\jump{u\circ\varphi}=\jump{\sigma_i^{-1}}$, that is the thesis.
\end{proof}

\begin{prop}\label{tecnical}
	Let $\nu,\eta\in C^{0,1}(\Su;\R^2\setminus \{P_1,\dots,P_n\})$ be such that $\gamma(\theta)=\eta(\theta')$ for some $\theta,\theta'\in \Su$. Then there exists a word $\beta\in F(n)$ such that 
	\begin{align}
		\jump{\nu\star\eta}=\jump{\nu\beta\eta\beta^{-1}}.
	\end{align}
\end{prop}
\begin{proof}
	Let $\alpha,\widehat \alpha:[0,\pi]\rightarrow \R^2\setminus \{P_1,\dots,P_n\}$ be two Lipschitz paths connecting $\nu(\theta)$ with $P$, in such a way that  the curves $\alpha^{-1}\star\nu\star\alpha$ and $\widehat \alpha^{-1}\overline\eta\star\widehat \alpha$ are based at $P$. In this way  $$\jump{\nu}=\jump{ \alpha^{-1}\star\nu\star\alpha},\qquad \jump{\eta}=\jump{\widehat \alpha^{-1}\star\eta\star\widehat\alpha}.$$
	Now, $\alpha^{-1}\star\widehat \alpha$ is based at $P$, so noting by $\beta$ a word representing it, we have
	$$\beta\eta\beta^{-1}\equiv(\alpha^{-1}\star\widehat \alpha)\star (\widehat \alpha^{-1}\star\eta\star\widehat\alpha)\star(\alpha^{-1}\star\widehat \alpha)^{-1}=\alpha^{-1}\star\eta\star\alpha.$$
	Hence
	$$\nu\beta\eta\beta^{-1}\equiv (\alpha^{-1}\star\nu\star\alpha)\star(\alpha^{-1}\star\eta\star\alpha)=\alpha^{-1}\star \nu\star \eta\star\alpha\in \jump{\nu\star\eta}.$$
\end{proof}

\begin{cor}\label{cor2.35}
	Assume that $\nu,\zeta\in C^{0,1}(\Su;\R^2\setminus \{P_1,\dots,P_n\})$ are such that there is a Lipschitz homothopy $\Psi:[0,1]\times \Su\rightarrow\R^2$ with $\Psi(0,\cdot)=\nu(\cdot)$ and $\Psi(1,\cdot)=\zeta(\cdot)$, and such that $\Psi(\cdot,\theta)\in \R^2\setminus \{P_1,\dots,P_n\}$ for some $\theta\in \Su$, then there is a word $\beta\in F(n)$ such that 
	\begin{align}\label{237}
		\jump{\zeta}=\jump{\nu\beta\eta\beta^{-1}},
	\end{align}
	where $\eta$ is a word representing\footnote{Here we are identifying again $[0,1]$ with $\Su$ when concatenating $\nu$ and $\eta$ with $\Psi(\cdot,\theta)$ } $\eta:=\nu^{-1}\star\Psi(\cdot,\theta) \star\zeta\star\Psi(\cdot,\theta)^{-1}$ (cf with Figure \ref{fig3:Corollary}).
\end{cor}
\begin{proof}
	Observe that
	\[
	\nu\star \eta=\nu\star \nu^{-1}\star\Psi(\cdot,\theta) \star\zeta\star\Psi(\cdot,\theta)^{-1}\in\jump{\Psi(\cdot,\theta) \star\zeta\star\Psi(\cdot,\theta)^{-1}}= \jump{\zeta}.
	\]
Then $\jump{\nu\star \eta}=\jump{\zeta}$.  Now Proposition \ref{tecnical} immediately implies the thesis.
\end{proof}
\begin{figure}
\begin{center}
\includegraphics[scale=0.6]{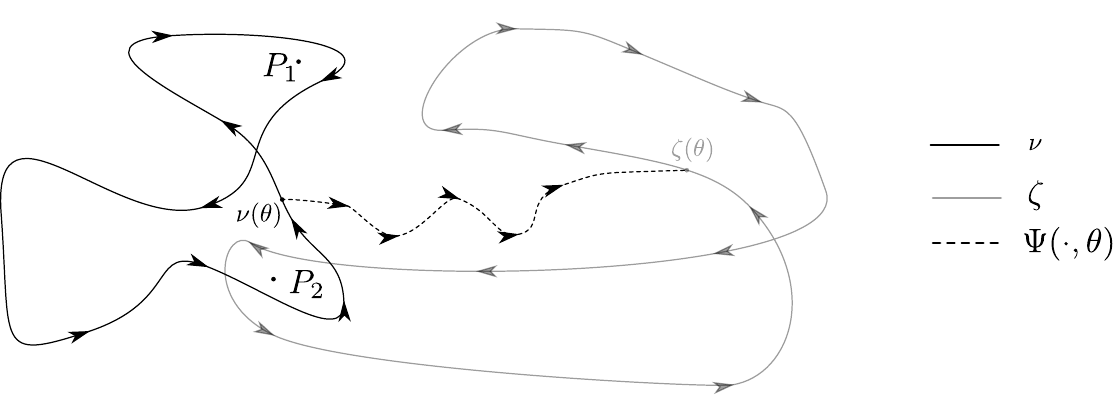}
\caption{A depiction of the construction of $\eta$ in Corollary \ref{cor2.35} in the case $n=2$: the function $\Psi$ identifies a curve $\Psi(\cdot,\theta):[0,1]\rightarrow\R^2\setminus \{P_1,P_2\}$ (the dotted line in the figure) connecting $\nu$ (the solid dark line) to $\zeta$ (the solid grey line) with starting and ending points $\Psi(0,\theta)=\nu(\theta)$,  $\Psi(1,\theta)=\zeta(\theta)$. The curve $\eta$ is then the suitable concatenation of $\nu$, $\zeta$ and $\Psi(\cdot,\theta)$ and their inverses.}\label{fig3:Corollary}
\end{center}
\end{figure}
\section{Proof of Theorem \ref{teo_main}}\label{sct:PfMain}

\subsection{Lower bound}\label{sbsct:LB}
In this section we show that 
$\mathcal A(\gamma)$ is larger or equal to the right-hand side of \eqref{eq_main}.

Let $u:D\rightarrow\R^2$ be a Lipschitz continuous map with $u\Big{|}_{\Su}=\gamma$.
We introduce the set
\[
\mathrm{Ad}(u):=\{(\#(u^{-1}(z_1)) ,  \ldots, \#(u^{-1}(z_n) )\in \mathbb{N}^n \ | \ z_i\in U_i\cap E_u\},
\]
where $E_u$ is defined in \eqref{eu} and we set $\#(\varnothing .  )=0$.
Then the following is in force:
\begin{lemma}\label{lem:Crucial}
For all $u:D\rightarrow\R^2$ Lipschitz functions with $u\Big{|}_{\Su}=\gamma$ it holds
\[
\mathrm{Ad}(u)\subset \mathrm{Ad}(\gamma).
\]
Moreover there exists an extremant $(k_1^*,\ldots,k_n^*)\in \mathrm{Ad}(u)$ such that for all $(k_1,\ldots,k_n)\in \mathrm{Ad}(u)$ it holds
\[
k_i^*\leq k_i \ \ \text{for all $i=1,\ldots,n$}.
\]
\end{lemma}
We show immediately how to derive the following Proposition starting from Lemma \ref{lem:Crucial}. 
\begin{prop}\label{prop:lowerbound}
For all Lipschitz curves $\gamma$ satisfying hypothesis (P) it holds
\[
\mathcal{A}(\gamma)\geq \min\left\{\left.\sum_{i=1}^n k_i |U_i|  \ \ \right| \ \ (k_1,\ldots,k_n)\in \mathrm{Ad}(\gamma) \right\}.
\]
\end{prop}
\begin{proof}
Let $u:D\rightarrow\R^2$ be a Lipschitz function with $u\Big{|}_{\Su}=\gamma$.  Let $(k_1^*,\ldots,k_n^*)\in \mathrm{Ad}(u)$ be the extremant.  Then $(k_1^*,\ldots,k_n^*)\in \mathrm{Ad}(\gamma)$ and $\mul(u,D,z)\geq k_i^*$ for all $z\in U_i$.  Therefore,  by the area formula
\begin{align*}
A(u)=&\int_D |\det (\nabla u)|\d x= \sum_{i=1}^n \int_{U_i} \mul(u,D,z)\d z\\
&\geq \sum_{i=1}^n k_i^* |U_i|\geq \min\left\{\left.\sum_{i=1}^n k_i |U_i|  \ \ \right| \ \ (k_1,\ldots,k_n)\in \mathrm{Ad}(\gamma) \right\}.
\end{align*}
\end{proof}
We now focus in the proof of Lemma \ref{lem:Crucial}.
\begin{proof}[Proof of Lemma \ref{lem:Crucial}]
Fix a family $(z_1,\ldots,z_n)\in U_1\times \ldots \times U_n$ such that for all $i=1,\dots,n$ we have $z_i\in E_u$.  Let $k_i:=\#(u^{-1}(z_i))$ and set $u^{-1}(z_i):=\{x_1^i, \ldots x_{k_i}^i\}$.   Let $\tilde{D}\subset D$ be a small disc contained in $D$ such that $ u^{-1}(z_i) \cap \tilde{D}=\varnothing .  $ for all $i=1,\ldots,n$.  Consider  an homothopy $\Phi:[0,1]\times \Su\rightarrow D$ between $\Su$ and $\partial \tilde{D}$.  Call $E_t\subset D$ the set bounded conneced component of $\R^2\setminus \Phi(t,\Su)$ (so that $E_0=D$ and $E_1=\tilde{D}$).  We design $\Phi$ so that $E_{t'}\subset E_t$ for all $t'\geq t$ and
\begin{align*}
 (1)& \ \Phi(t,\Su) \subset D\setminus \tilde{D} \ \ \ \text{for all $t\in (0,1)$};\\
(2) &\ \# (u^{-1}(z_i)\cap  \Phi(t,\Su))\leq 1\ \ \  \text{for all $i=1,\ldots,n$,  $t\in (0,1)$};\\
(3)& \ \text{if $u^{-1}(z_i)\cap  \Phi(t,\Su)\neq \varnothing .  $} \ \ \Rightarrow \ \ u^{-1}(z_j)\cap  \Phi(t,\Su)= \varnothing .   \ \text{for all $j\neq i$}.
\end{align*}
Let 
\begin{align*}
\tau:=\sup\left\{t\in [0,1] \ \left| \ \left(\bigcup_{i=1}^n u^{-1}(z_i) \right) \subset E_t \right. \right\}.
\end{align*}
\begin{figure}
\begin{center}
 \includegraphics[scale=0.76]{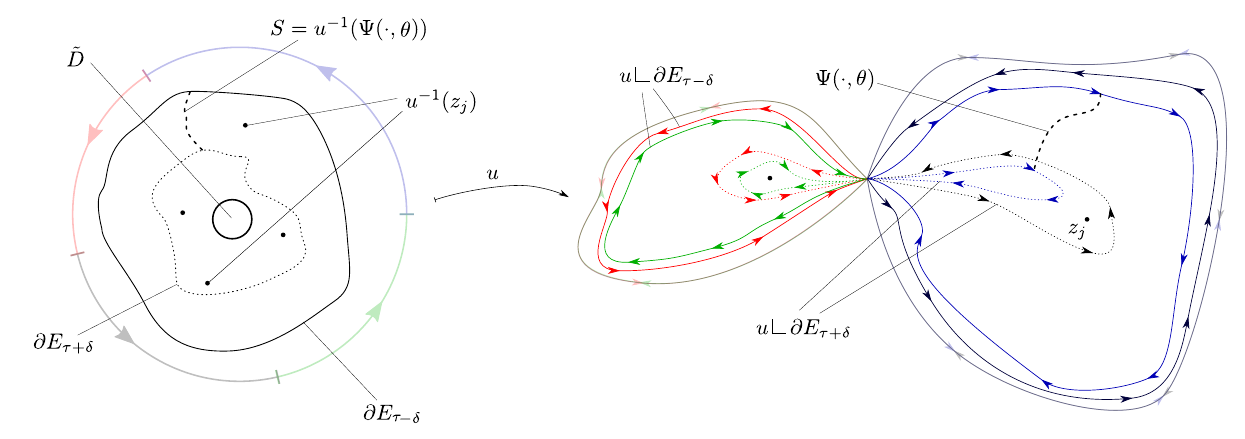}
\caption{A depiction of the argument in the proof of the lower bound,  Lemma \ref{lem:Crucial},  applied to the case of Figure \ref{fig:Example}.  Notice that $\tau$ is the first time in which $\partial E_{\tau}=\Phi(\tau,\Su)$ (not depicted here) hits $u^{-1}(z_j)$ and thus, by construction of $\Phi$: $\#(u^{-1}(z_j)\cap (E_{\tau-\delta}\setminus E_{\tau+\delta}))=1$ for some $\delta>0$.}\label{fir:LemCrucial}
\end{center}
\end{figure}
The way we have designed $\Phi$ (condition $(2)-(3)$) implies that,  for $t=\tau$,  there is a unique 
\[
x\in\left(\bigcup_{i=1}^n u^{-1}(z_i) \right)\cap \Phi(\Su,\tau).
\]  
So let $x\in u^{-1}(z_j) $ be such a unique point (for some $j$).  By continuity of the homothopy for some small $\delta>0$ we have $x\in E_{\tau-s}\setminus E_{\tau+s}$ (cf with Figure \ref{fir:LemCrucial}) for all $s<\delta$ and
\[
\left(\bigcup_{i=1}^n u^{-1}(z_i) \right) \setminus \{x\} \subset E_t \ \ \text{for $\tau<t<\tau+\delta$}.
\]
Call $\gamma_1:=u\res\partial E_{\tau+\delta}=u\res \Phi(\Su,\tau+\delta)$.  First we note that $u\res \partial E_{\tau-\delta}\in \jump{\gamma}$ by Lemma \ref{teo_2.18}.  We now want to apply Corollary \ref{cor2.35} with $\nu=u\res \partial E_{\tau-\delta}$ and $\zeta=\gamma_1$.  Observe that $u\res \partial E_{\tau-\delta}$ and $\gamma_1$ are linked by the Lipscthitz homotopy $\Psi:[0,1]\times \Su\rightarrow \R^2$ defined as 
\[
\Psi(t, \xi):=u\left(\Phi\left(t(\tau+\delta)+(1-t)(\tau-\delta),\xi\right)\right).
\] 
Clearly, by how $\Phi$ has been designed,  we can find $\theta\in \Su$ such that $\Psi(t,\theta)\in \R^2\setminus \{z_1,\ldots,z_n)$ for all $t\in [0,1]$.   Still by construction of $\Phi$ we can choose $\theta$ to be also such that
\[
S:=\{u^{-1}(\Psi(t,\theta)) \ | \ t\in[0,1]\}=\{\Phi(t(\tau+\delta)+(1-t)(\tau-\delta),\theta) \ | \ t\in [0,1]\}
\]
is a Lipscthitz curve inside $E_{\tau-\delta}\setminus E_{\tau+\delta}$.  Then we can apply Corollary \ref{cor2.35} to deduce that 
$$\jump{\gamma_1}=\jump{\gamma\beta\eta\beta^{-1}}$$
with \footnote{Up to the usual identification between $[0,1]$ with $\Su$.} $\eta:=\nu^{-1}\star\Psi(\cdot,\theta) \star\zeta\star\Psi(\cdot,\theta)^{-1}$.  Set now $U:=(E_{\tau+\delta}\setminus E_{\tau-\delta})\setminus S$ and observe that 
\begin{equation}
\mathrm{mul}(u,U,p)=\left\{
\begin{array}{ll}
1 & \ \text{if $p=z_j$}\\
0 & \ \text{if $p=z_i$ for $i\neq j$}
\end{array}\right.
\end{equation}
and since
$\eta=u\res \partial U$,  by applying Lemma \ref{lem_2.33} we conclude that  $\eta\in \jump{\sigma_j}$  or $\eta\in \jump{\sigma_j^{-1}}$.  Then it holds that either 
\[
\jump{\gamma_1}=\jump{\gamma\beta\sigma_j\beta^{-1}} \ \ \ \text{or} \ \ \ \ \jump{\gamma_1}=\jump{\gamma\beta\sigma_j^{-1}\beta^{-1}}.
\]
In particular $\gamma_1$ is a $\mathbf{e}_j$-injection of $\gamma$.  \\

By restarting the process on $\gamma_1$ we move to the next point $y\in \left(\bigcup_{i=1}^n u^{-1}(z_i) \right) \setminus \{x\} $.  Say $y\in u^{-1}(z_k)$.  By repeating this argument we produce a $\gamma_2$ which is a $\mathbf{e}_k$-injection of $\gamma_1$.  By Remark \ref{rmk:sumofinj} then $\gamma_2$ is a $\mathbf{e}_k+\mathbf{e}_j$ injection of $\gamma$.  We keep repeating this argument until we reach $\partial \tilde{D}$.  Since each new curve produced by trepassing a counterimage gives an injection,  and since $u(\partial \tilde{D})\equiv 1$ it follows that the null curve is exactly a $(k_1,\ldots,k_n)$-injection of $\gamma$.  This yields immediately that $(k_1,\ldots,k_n)\in \mathrm{Ad}(\gamma)$.\\
\smallskip
The extremant is now easily constructed by selecting for each $i=1,\ldots,n$
\[
k_i^*:=\mathrm{arg}\min\{\#(u^{-1}(z_i)) \ | \ z_i\in U_i\}.
\]
\end{proof}

Here it is worth observing that the same proof applies and leads to the following Corollary of Proposition \ref{prop:lowerbound}, which is valid for any Lipschitz curve (not necessarily satisfying hypothesis (P)):

\begin{cor}\label{cor:lowerbound}
Let $\gamma:\Su\rightarrow \R^2$ be a Lipschitz curve and let $U_i$, $i=1,\dots,m$ be some (not necessarily all) of the bounded connected components of $\R^2\setminus \Gamma$. Then 
\[
\mathcal{A}(\gamma)\geq \min\left\{\left.\sum_{i=1}^m k_i |U_i|  \ \ \right| \ \ (k_1,\ldots,k_m)\in \mathrm{Ad}(\gamma_{(m)}) \right\},
\]
where $\gamma_{(m)}\in F(m)$ is a representative of $\gamma$ in the homotopy group $\pi_1(\R^2\setminus \{P_1,\dots,P_m\})$, $P_i\in U_i$, $i=1,\dots,m$.
\end{cor}

\subsection{Upper bound}\label{sbsct:UB}
We are left with proving that $\leq $ holds in \eqref{eq_main}.  To do that we need to choose a precise sets of generator of $\pi_1(\R^2\setminus \{P_1,\ldots,P_n\})$: the curves supported on the boundaries of the connected component of $\R^2\setminus\Gamma$.  Thus Remark \ref{rmk:NoDepOnGen} becomes relevant in what follows.\\

Let $\gamma:\Su\rightarrow \R^2$ be a curve as in Theorem \ref{teo_main}; for all bounded connected components $U_i\subset\R^2\setminus \Gamma$, $i=1,\dots,n$, we observe that $U_i$ is simply connected and $\mathcal H^1(\partial U_i)<+\infty$, as $\partial U_i\subseteq\Gamma$. We choose a Lipschitz curve $ \tau_i$ parametrizing in counterclockwise order $\partial U_i$, and compose it with a $1$-null curve $\omega_i$ based at $P$ and reaching $\partial U_i$. The obtained curve $\overline \sigma_i:=\overline \tau_i\star\omega_i$ is based at $P$ and satisfies
$$\text{\rm link}(\overline\sigma_i ,P_j)=\text{\rm link}( \tau_i,P_j)=\delta_{ij}\qquad \qquad \forall j=1,\dots,n.$$
Up to change the $1$-null curve connecting $\partial U_i$ to $P$ we can always assume that $\overline \sigma_i $ is homotetic to $\sigma_i$ in Definition \ref{def_sigma} and in particular $\overline \sigma_i\equiv\sigma_i$ in $\pi_1(\R^2\setminus\{P_1,\dots,P_n\})$.  It is convenient to introduce the new basis
\begin{align}\label{def_Sigma2}
	\overline\Sigma(n):=\{\overline \sigma_1,\dots,\overline \sigma_n,\overline \sigma_1^{-1},\dots,\overline \sigma_n^{-1}\}.
\end{align}
Let now assume that $\Gamma\subset\R^2$ is a Jordan curve and let $S$ be the unique bounded connected component of $\R^2\setminus \Gamma$. Then the existence of a classical solution to the Plateau problem tells us that there is a map $u:D\rightarrow \R^2$ which belongs to $H^1(\text{{\rm int}}(D);\R^2)\cap C^0(D;\R^2)$,  which is harmonic,  conformal in $ \text{{\rm int}}(D)$ and minimizes the area functional $A$.  Furthermore,  the classical theory of currents implies that the integral current $u_\sharp\jump{D}$ coincides with the integration over $S$. In particular one infers $A(u)\geq |u_\sharp\jump{D}|=|S|$. On the other hand, if $\overline u:D\rightarrow \overline S$ is the Riemann map, namely a biholomorphic bijection  between $D$ and $\overline S$, then $A(\overline u)=|S|$. Hence, $\overline u\in H^1(\text{{\rm int}}(D);\R^2)\cap C^0(D;\R^2)$ is a solution to the Plateau problem. This argument leads us to the following:

\begin{prop}\label{Riemann_plateau}
	Let $\gamma \in  C^{0,1} (\Su;\R^2)$ be an injective closed curve in $\R^2$, and let $S$ denote the unique bounded connected component of $\R^2\setminus \Gamma$. Then
\[
\mathcal A(\gamma)=|S|.
\]
\end{prop}
\begin{proof}
	Let $u\in H^1(\text{{\rm int}}(D);\R^2)\cap C^0(D;\R^2)$ be a solution to the Plateau problem for the boundary $\Gamma=\gamma(\Su)$, so that, from the previous discussion, we have $A(u)=|S|$. Furthermore,  since $u$ is the minimizer on a larger class than just Lipscthitz maps,  trivially we have $\mathcal A(\gamma)\geq A(u)$.  Let us prove the opposite inequality,  and  consider the maps $u_n:D\rightarrow \R^2$ defined as 
\[
u_n(x):=u\left( \frac{n-1 }{n}x \right)\qquad \qquad \forall  x\in D,
\]
	for all $n\geq2$. Also, set $\gamma_n:\Su\rightarrow \R^2$ as $\gamma_n=u_n\res \Su$. Therefore, by  change of variable,  
	\[A(u_n)=\int_D|\det(\nabla u_n)|dx=\int_{B_{1-\frac1n}(0)}|\det(\nabla u)|dx\rightarrow A(u)=|S|,
	\]
	as $n\rightarrow\infty$. On the other hand, $A(u_n)\geq \mathcal A(\gamma_n)$, and since $\gamma_n\rightarrow \gamma$ uniformly, thus in the Frechet metric, we conclude
	$$\mathcal A(\gamma)=\lim_{n\rightarrow \infty}\mathcal A(\gamma_n)\leq \lim_{n\rightarrow \infty}A(u_n)=|S|,$$
	by Proposition \ref{prop:continuity}.
\end{proof}
We are now in a position to prove the upper bound:
\begin{prop}\label{prop:upperbound}
	For all $\gamma\in F(n)$ it holds
	\[
	\mathcal{A}(\gamma)\leq \min\left\{\left.\sum_{i=1}^n k_i |U_i|  \ \ \right| \ \ (k_1,\ldots,k_n)\in \mathrm{Ad}(\gamma) \right\}
	\]
\end{prop}
\begin{proof}
%
Let $(k_1,\ldots,k_n)\in \mathrm{Ad}(\gamma)$ so  there is a $(k_1,\ldots,k_n)$-injection  of $\gamma$ which is equivalent to the null word.  In particular, by definition of injection and by Remark \ref{rem_attheend}, we can find a sequence of monoids $\beta_j\alpha_j\beta_j^{-1}$, with $j=1,\dots,K$, $K:=k_1+\ldots+k_n$, and $\alpha_j\in \overline\Sigma(n)$, such that 
$$1=\gamma(\beta_1\alpha_1\beta_1^{-1})(\beta_2\alpha_2\beta_2^{-1})\dots (\beta_K\alpha_K\beta_K^{-1})\qquad \text{ in }F(n).$$
This implies 
$$\gamma=(\beta_K\alpha_K\beta_K^{-1})^{-1}\dots (\beta_2\alpha_2\beta_2^{-1})^{-1}(\beta_1\alpha_1\beta_1^{-1})^{-1}.$$
Hence from Lemma \ref{lem_subadd} it follows
$$\mathcal A(\gamma)\leq \sum_{j=1}^K\mathcal A((\beta_j\alpha_j\beta_j^{-1})^{-1})=\sum_{j=1}^K\mathcal A(\alpha_j^{-1})=\sum_{i=1}^n k_i|U_i|.$$
In the first equality we have used Lemma \ref{teo_Ageneric}, and in the latter,  Proposition \ref{Riemann_plateau}.  The thesis follows by the arbitrariness of $(k_1,\ldots,k_n)\in \mathrm{Ad}(\gamma)$.
\end{proof}

\section{The case of infinitely many connected components of $\R^2\setminus \Gamma$}\label{sct:Cnn}
In this section we generalize Theorem \ref{teo_main} to any Lipschitz curve $\gamma:\Su\rightarrow \R^2$.
\medskip

For such a general Lipschitz curve, setting as usual $\Gamma:=\gamma(\Su)$, we know that $\R^2\setminus \Gamma$ consists of at most countable many open connected components. Let us call $(U_i)_{i\geq 0}$ these connected components, which are bounded (and simply connected) unless $i=0$, corresponding to $U_0$, the unique unbounded connected component. For all $i\geq1$ we pick a point $P_i\in U_i$. \\

For all $N>0$ we consider the homotopy group $\pi_1(\R^2\setminus \{P_1,\dots,P_N\})=F(N)$; the curve $\gamma$ has an indecomposable representation in $F(N)$, for all $N>0$, that we denote by $\gamma_{(N)}$. 
We denote by 
\begin{align}
	\mathcal A^N(\gamma):=\min\left\{\left.\sum_{i=1}^N k_i |U_i|  \ \ \right| \ \ (k_1,\ldots,k_N)\in \mathrm{Ad}\left(\gamma_{(N)}\right) \right\}.
\end{align}
The following theorem then holds:

\begin{theorem}\label{teo_main2}
	Let $\gamma:\Su \rightarrow \R^2$ be a Lipschitz curve, let $\tau:\mathbb N\rightarrow \mathbb N$ be a bijection, and let $\gamma_{(N)}^\tau\in F(N)$ represent $\gamma$ in $\pi_1(\R^2\setminus \{P_{\tau(1)},\dots,P_{\tau(N)}\})$. Then 
\begin{equation}
\lim_{N\rightarrow \infty}\mathcal A^N(\gamma)=\lim_{N\rightarrow \infty}\left(\min\left\{\left.\sum_{i=1}^N k_{\tau(i)} |U_{\tau(i)}|  \ \ \right| \ \ \begin{array}{c}
(k_1,\ldots,k_N)\in \mathrm{Ad}\left(\gamma_{(N)}^\tau\right)
\end{array} \right\}\right),
\end{equation}
and moreover
\begin{align}\label{eq_main2}
	\mathcal{A}(\gamma)=\lim_{N\rightarrow \infty}\mathcal A^N(\gamma).
\end{align}
\end{theorem}

Before proving this, we anticipate the following approximation Lemma:

\begin{lemma}\label{lem_approxcurves}
	Let $\gamma:\Su\rightarrow \R^2$ be a Lipschitz curve; then for all $\varepsilon>0$ there exists a smooth curve  $\gamma_\varepsilon:\Su\rightarrow \R^2$ such that, noting $\Gamma_\varepsilon:=\gamma_\varepsilon(\Su)$, one has 
	\begin{itemize}
		\item[(a)] $\R^2\setminus \Gamma_\varepsilon$ consists of finitely many connected components;
		\item[(b)] $\|\gamma_\varepsilon-\gamma\|_{L^\infty(\Su;\R^2)}\leq \varepsilon$;
		\item[(c)] $\int_{\Su}|\dot\gamma_\varepsilon(t)|dt\leq\int_{\Su}|\dot\gamma(t)|dt+\varepsilon$. 
	\end{itemize}
\end{lemma}

\begin{proof}
	By mollification, for all $\varepsilon>0$ we can find a smooth curve $\eta:\Su\rightarrow \R^2$ so that $$\|\eta-\gamma\|_{L^\infty(\Su;\R^2)}\leq \varepsilon/2\qquad \text{and }\qquad \int_{\Su}|\dot\eta(t)|dt\leq\int_{\Su}|\dot\gamma(t)|dt+\varepsilon/2.$$
	For all $t,s\in [0,2\pi)$, with $t\neq s$, we define (we identify $\Su$ with $[0,2\pi)$)
	$$\delta_j(t,s)=\frac{\eta_j(t)-\eta_j(s)}{s-t},\qquad j=1,2.$$
	Since $\eta$ is smooth, it turns out that $\delta_j$ extends to the set $\{t=s\}$ and $\nabla \delta_j\in L^\infty([0,2\pi)\times[0,2\pi),\R^2)$; in particular
	$\int_{[0,2\pi)^2}|\det\nabla \delta(t,s)|dtds$ is finite and by area formula $\text{{\rm mul}}(\delta,[0,2\pi)^2,\cdot)\in L^1(\R^2)$. This implies that $\text{{\rm mul}}(\delta,[0,2\pi)^2,\lambda)$ is finite for a.e. $\lambda\in \R^2$. We can then select $\lambda\in \R^2\setminus \{0\} $ arbitrarily close to $0$ in such a way that $\delta^{-1}(\lambda)$ is finite; specifically, we can choose $|\lambda|<\varepsilon/16\pi$ so that, defining
	$\zeta(t):=\eta(t)+\lambda t,$ for  $t\in[0,2\pi),$
	it turns out 
	$$\|\zeta-\eta\|_{L^\infty(\Su;\R^2)}\leq \varepsilon/4\qquad \text{and }\qquad \int_{\Su}|\dot\zeta(t)|dt\leq\int_{\Su}|\dot\eta(t)|dt+\varepsilon/4.$$
	Let us now show that the number of self-intersections of $\zeta$ is finite. Indeed, assume that $\zeta(t)=\zeta(s)$, then by definition it follows $\eta(t)-\eta(s)=\lambda(s-t)$. Hence $\delta(t,s)=\lambda$ and $(t,s)\in \delta^{-1}(\lambda)$. Therefore the number of couples $(t,s)$ for which $\zeta(t)=\zeta(s)$ is finite. 
	
	Eventually, we design a path $\widetilde\zeta$ from $\zeta(2\pi)$ to $\zeta(0)$ in such a way that its length is smaller than $\varepsilon/4$, it intersects the image of $\zeta$ at most finitely many times, and $\gamma_\varepsilon:=\zeta\star\widetilde \zeta$ is smooth and satisfies (b) and (c). 
	
	It remains to show that $\R^2\setminus \Gamma_\varepsilon$ decomposes in finitely many connected components. We do that by induction on the number
	$$N_\varepsilon:=\#\{(t,s)\in [0,2\pi)^2:t< s,\;\gamma_\varepsilon(t)=\gamma_\varepsilon(s)\}.$$ 
	If $N_\varepsilon=0$ then $\gamma_\varepsilon$ parametrizes a Jordan curve and $\R^2\setminus \Gamma_\varepsilon$ consists of two connected components.
	Assume the thesis is true for $N_\varepsilon\leq N$, for some $N\in \mathbb N$, and consider the case $N_\varepsilon=N+1$. Take a couple $(t,s)$ such that $s-t$ is minimal among all such couples, and so the curve $\gamma_\varepsilon\res [t,s)$ is injective. The curve $\gamma_\varepsilon\res [0,t)\star\gamma_\varepsilon\res[s,2\pi)$ is a closed (Lipschitz) loop whose corresponding number $N_\varepsilon$ is smaller than $N+1$, so its image $\widehat \Gamma_\varepsilon$ divides $\R^2$ in finitely many connected components. Now, by hypothesis on $\gamma_\varepsilon$, the set
	$$\{r\in [t,s):\gamma_\varepsilon(r)\in \widehat \Gamma_\varepsilon\}=\{r_1,r_2,\dots,r_{M_\varepsilon}\}$$
	is finite. So, assuming $r_i<r_{i+1}$, $\forall i=1,\dots,M_\varepsilon-1$ in the preceding expression, 
	every curve $\gamma_\varepsilon\res [r_i,r_{i+1})$ is contained in a given connected component of $\R^2\setminus \widehat \Gamma_\varepsilon$, and divides it into two components. Since this is true for every $i=1,\dots, M_\varepsilon-1$, the number of connected components of $\R^2\setminus \Gamma_\varepsilon$ remains bounded, and item (a) is achieved.
\end{proof}
We also need the following 

\begin{prop}\label{prop:crucial}
	Let $\eta:\mathbb S^1\rightarrow \R^2$ a Lipschitz curve satisfying Property (P); let $V_1,\dots,V_N$ be the bounded connected components of $\R^2\setminus \eta(\Su)$, and let $\mathcal R$ be a finite family of points in $\R^2\setminus \eta(\Su)$ such that for all $j=1,\dots,N$
	$$\{X_{ij}\in \mathcal R\}=\mathcal R\cap V_j\neq \varnothing.$$
	Let $\eta^{\mathcal R}$ be a word representing $\eta$ in $\pi_1(\R^2\setminus \mathcal R)$, and assume that $(\dots,\widehat k_{ij},\dots)\in \mathbb N^{\#(\mathcal R)}$
is a string in $\text{{\rm Ad}}(\eta^{\mathcal R})$. Then $(k'_1,\dots,k'_N)\in \text{{\rm Ad}}(\eta)$, where  
$k'_j:=\min\{\widehat k_{ij}:X_{ij}\in \mathcal R\cap V_j\}.$		
%
%
\end{prop}
\begin{proof}
For all $j=1,\dots,N$, let $R_j:=R_{i_jj}\in \mathcal R\cap V_j$ be any point (for some index $i_j$), and denote by $\overline{\mathcal R}:=\{R_1,\dots, R_N\}\subset\mathcal R$.   Since there is trivially an isomorphism between $\pi_1(\R^2\setminus \overline{\mathcal R})$ and  $\pi_1(\R^2\setminus \{P_1,\dots, P_N\})$ we can identify the two groups (and in particular the generators) to be both $F(N)$.  Let $\pi^{\mathcal R}: \pi_1(\R^2\setminus {\mathcal R})\rightarrow \pi_1(\R^2\setminus \overline{\mathcal R})$ be the application that associates to a word $w$ of $F(\#(\mathcal R))$ the word $\pi^{\mathcal R}(w)\in F(N)$ obtained from $w$ by erasing all the symbols which belongs to the generators of $F(\#(\mathcal R))$ but not on that of $F(N)$.  The application is a morphism since clearly $\pi^{\mathcal{R}} (\omega \tau)= \pi^{\mathcal{R}} (\omega )\pi^{\mathcal{R}} (\tau)$.\\

 Thus, assuming $(\dots,\widehat k_{ij},\dots)\in \text{{\rm Ad}}(\eta^{\mathcal R})$, we get a sequence of injections $\beta_r\sigma_r\beta_r^{-1}$ such that 
 $$\eta\beta_{1}\alpha_{1}\beta_{1}^{-1}\beta_{2}\alpha_{2}\beta_{2}\dots\beta_{H}\alpha_{H}\beta_{H}=1, \ \ \ H=\sum_{i,j} k_{ij}, \ \ \alpha_i\in \Sigma(\#(\mathcal{R})).$$
 But since $\pi^{\mathcal R}(1)=1$, we also have
 $$\overline \eta\overline\beta_{1}\overline\alpha_{1}\overline\beta_{1}^{-1}\overline\beta_{2}\overline\alpha_{2}\overline\beta_{2}^{-1}\dots \overline\beta_{H}\overline\alpha_{H}\overline\beta_{H}^{-1}=1,$$
 where,  to shortcut the notation we have denoted $\overline \zeta:=\pi^{\mathcal R}(\zeta)$. On the other hand the injections $\overline \beta_r\overline \alpha_r\overline \beta_r^{-1}$ are null if $\alpha_r$ is a generator of $F(\#( \mathcal R))$ but not  of $F(N)$,  and coincides with $\alpha_r$ otherwise. This means that 
 $$(\dots, k_{j},\dots )\in \text{{\rm Ad}}(\eta),$$
where $ k_j:=k_{i_jj}$. In particular the thesis follows if we choose $i_j$ so that $k_{i_jj}=\min\{k_{ij}:X_{ij}\in \mathcal R\cap V_j\}$.
\end{proof}

%
%
%
%

The last ingredient needed to prove Theorem \ref{teo_main2} is the following technical Lemma.

\begin{lemma}\label{Rem_34}
Let $\nu$, $\zeta$ and  $\Psi$ be as in  Corollary \ref{cor2.35}.  Assume also that the $\theta$ is chosen in a way that the map $\Psi(\cdot,\theta):[0,1] \rightarrow  \R^2$ is bi-Lipstchitz and injective. Then if $P_i\in E_{\Psi}$ for all $i=1,\ldots,N$ it holds that
	$$(\mul(\Psi,A,P_1)+k_1,\dots,\mul(\Psi,A,P_n)+k_n)\in \textrm{Ad}(\zeta),$$
	whenever $(k_1,\dots,k_n)\in \textrm{Ad}(\nu)$.
\end{lemma}
\begin{proof}
Let us denote by $A:=[0,1]\times \Su$ the domain of $\Psi$, and let $\eta:=\nu^{-1}\star\Psi(\cdot,\theta) \star\zeta\star\Psi(\cdot,\theta)^{-1}$ be the curve given by Corollary \ref{cor2.35}.  Now since the map $\Psi(\cdot,\theta)$ is bi-Lipscthitz and injective we can find a diffeomorfism $\Phi: D \rightarrow ([0,1]\times \Su)\setminus ([0,1]\times \{\theta\})$ (see Figure \ref{clarifig}) such that 
\[
\Phi(\Su)=\Phi(\partial D)=(\{0\}\times \Su)\cup (\{1\}\times \Su) \cup  ([0,1]\times\{\theta\}).
\]  
\begin{figure}
\begin{center}
\includegraphics[scale=0.7]{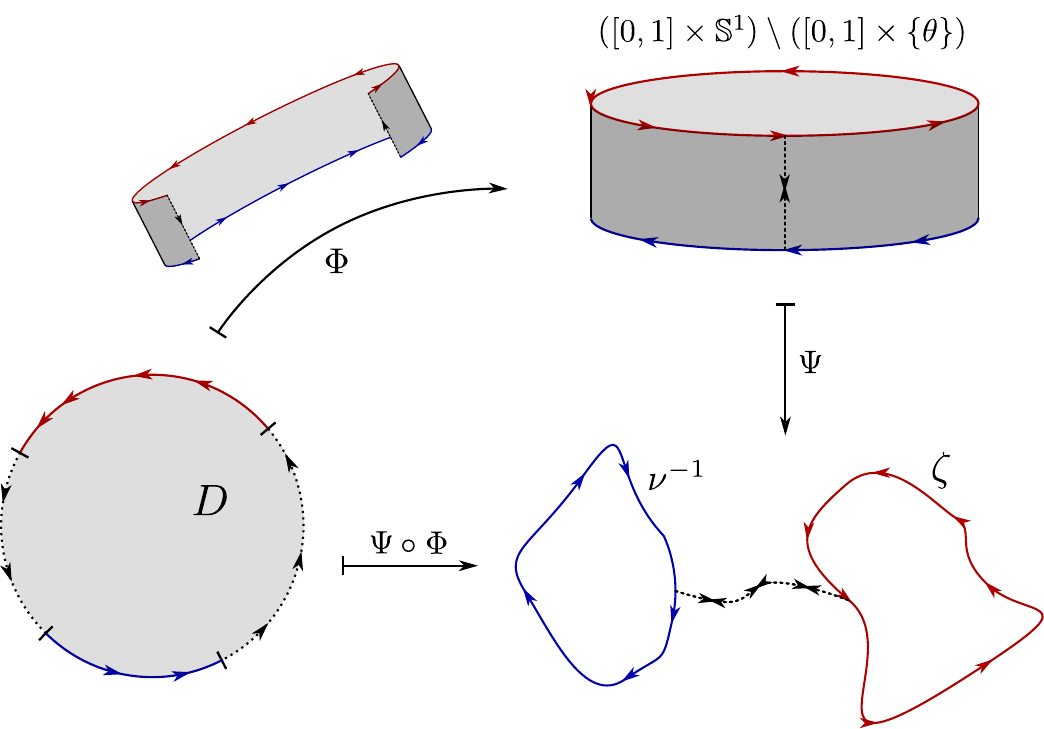}
\end{center}
\caption{A depiction of the diffeomorfism used to prove Lemma \ref{Rem_34}: we can build $\Phi$ so that $\Psi\circ \Phi\res \Su =\eta$ in order to apply Lemma \ref{lem:Crucial} to $\Psi\circ \Phi$.  here we are representing $[0,1]\times \Su$ as a subset in $\R^3$.  }\label{clarifig}
\end{figure}
Observe now that we can build $\Phi$ so that $\Psi\circ \Phi:\Su\rightarrow \R^2$ satisfies also $\Psi\circ \Phi (\Su)=\eta$.  Now since $\Phi$ is a diffeomorphism we have $\mathrm{mul}(\Psi\circ \Phi ,  D,z)=\mathrm{mul}( \Psi,  A,z)$ and,  by Lemma \ref{lem:Crucial} we can infer that $(\mathrm{mul}(\Psi\circ \Phi ,  D,P_1), \ldots, \mathrm{mul}(\Psi\circ \Phi ,  D,P_n))\in \mathrm{Ad}(\eta)$.  Then,  by  \eqref{237} 
	$$(\mul(\Psi,A,P_1)+k_1,\dots,\mul(\Psi,A,P_n)+k_n)\in \textrm{Ad}(\zeta),$$
	whenever $(k_1,\dots,k_n)\in \textrm{Ad}(\nu)$.
\end{proof}
We are now ready to prove Theorem \ref{teo_main2}.

\begin{proof}[Proof of Theorem \ref{teo_main2}]
	\textit{Step 1:}
On the one hand Corollary \ref{cor:lowerbound} implies that, for all $N\geq1$ it holds
$$\mathcal A(\gamma)\geq \mathcal A^N(\gamma).$$
We now show that for all $N\geq1$ we have $\mathcal A^{N+1}(\gamma)\geq \mathcal A^N(\gamma)$, so that existence of the limit in \eqref{eq_main2} is granted, and the previous inequality entails
$$\mathcal A(\gamma)\geq \lim_{N\rightarrow\infty}\mathcal A^N(\gamma).$$
Given a word $\eta$ in $F(N+1)$, we denote by $\Pi_N(\eta)$ the word in $F(N)$ obtained by erasing all the symbols $\sigma_{N+1}$ and $\sigma_{N+1}^{-1}$ from $\eta$.

Now it can be easily seen that if the null word is a $(k_1,\dots,k_N,k_{N+1})$-injection of $\eta$, then the null word is a $(k_1,\dots,k_N,0)$-injection of $\Pi_{N}(\eta)$.
As a consequence for every string $(k_1,\dots,k_N,k_{N+1})\in  \mathrm{Ad}(\eta)$, it holds $(k_1,\dots,k_N)\in \mathrm{Ad}(\Pi_N(\eta))$, and thus since $\sum_{i=1}^N k_i |U_i|\leq \sum_{i=1}^{N+1} k_i |U_i|$, this implies $\mathcal A^{N+1}(\gamma)\geq \mathcal A^N(\gamma)$.
It remains to prove that 
\begin{align}\label{reverse}
\mathcal A(\gamma)\leq \lim_{N\rightarrow\infty}\mathcal A^N(\gamma).
\end{align}
The next steps are devoted to this aim.

\textit{Step 2:} 
We fix $\varepsilon>0$ and use Lemma \ref{lem_approxcurves} to find a smooth curve $\gamma_\varepsilon$ satisfying (a), (b), and (c). Thanks to (b) and (c), we build a Lipschitz homotopy $\Psi_\varepsilon:[0,1]\times \Su\rightarrow \R^2$ between $\gamma$ and $\gamma_\varepsilon$ (e.g., by affine interpolation), in such a way that
\begin{align}\label{Ce}
\int_{\R^2}\text{{\rm mul}}(\Psi_\varepsilon,[0,1]\times \Su,z)dz=\int_0^1\int_{\Su}|\det\nabla \Psi_\varepsilon|dtds\leq C\varepsilon.
\end{align}
If $U_i$, $i\geq1$, denote the bounded connected components of $\R^2\setminus \Gamma$ ($U_0$ will denote the unbounded one), we consider the function $m_\varepsilon(z):= \text{{\rm mul}}(\Psi_\varepsilon,[0,1]\times \Su,z)$ and  select points $Q_i\in (U_i\cap E_{\Psi_\varepsilon})\setminus \Gamma_\varepsilon$, for all $i\geq1$, such that $$m_\varepsilon(Q_i)=\min_{z\in (U_i\cap E_{\Psi_\varepsilon})\setminus \Gamma_\varepsilon}\{m_\varepsilon(z)\}.$$
It will be useful to choose also a point $Q_0\in U_0$ with null multiplicity $m_\varepsilon(Q_0)=0$. Moreover, we denote  by $(V_j)_{j=0}^{M_\varepsilon}$ the connected components of $\R^2\setminus \Gamma_\varepsilon$ (where $V_{0}$ denotes the unbounded one) and  we denote, for all $i\geq 0$ and $j=0,\dots,M_\varepsilon$,
 $$W_{ij}:=U_i\cap V_j.$$
 Some of these sets might be empty. 
 For all $i\geq 0$ and $j=0,\dots, M_\varepsilon$  it is necessary to choose points $R_{ij}\in W_{ij}$ (for those $W_{ij}\neq \varnothing$) as follows:
 \begin{align}
R_{ij}:=\begin{cases}
		Q_i&\text{ if }Q_i\in V_j,\\
		\text{{\rm argmin}}\{m_\varepsilon(z):z\in W_{ij}\cap E_{\Psi_\varepsilon}\}&\text{ otherwise},
\end{cases}
 \end{align}
 where the second choice is arbitrary, whenever possible (if $W_{ij}\neq \varnothing)$.  Finally, for all $i\geq 0$ and $j=0,\dots,M_\varepsilon$, we denote 
 $$m_{ij}:=m_\varepsilon(R_{ij}),\qquad \qquad\text{if }W_{ij}\neq\varnothing,$$
and  $m_{ij}=0$ otherwise. Notice that $m_\varepsilon(R_{00})=0$ and so $m_{00}=0$.
%
%
%
%

 Now, we know that 
 \begin{align}\label{Aeps}
 	\mathcal A(\gamma_\varepsilon)=\sum_{j=1}^{M_\varepsilon}k_j|V_j|,
 \end{align}
 for some $(k_1,\dots,k_{M_\varepsilon})\in \text{{\rm Ad}}(\gamma_\varepsilon)$. We set $$K:=\max\{1,k_1,\dots,k_{M_\varepsilon}\}\qquad \qquad H:=\min\{|V_1|,\dots,|V_{M_\varepsilon}|\};$$ then we fix $N\in \N$ big enough so that 
 \begin{align}\label{48}
\sum_{i>N}|U_i|\leq\varepsilon\min\left\{ \frac{1}{K},\frac{H}{2}\right\}.
 \end{align}
Finally, for all $j=0,\dots,M_\varepsilon$,  we also set  $\widetilde W_{(N+1)j}:= \cup_{h> N}{W_{hj}}$.
 From \eqref{Ce} we deduce that 
 \begin{align}\label{45}
 	\sum_{j=0}^{M_\varepsilon}\sum_{i\geq1}m_{ij}|W_{ij}|= \sum_{j=0}^{M_\varepsilon}\sum_{i\geq1}m_\varepsilon(R_{ij})|W_{ij}|\le \sum_{j=0}^{M_\varepsilon}\sum_{i\geq1}\int_{W_{ij}}\text{{\rm mul}}(\Psi_\varepsilon,[0,1]\times \Su,z)dz\leq  C\varepsilon.
 \end{align}
%
%
%
%
  Let us compute $\mathcal A^N(\gamma)$; let  $\gamma^{(N)}$ be a word representing $\gamma$ in $\pi_1(\R^2\setminus \{Q_1,\dots,Q_{N}\})=F(N)$. Then 
 \begin{align}\label{aN}
 	\mathcal A^N(\gamma)=\sum_{i=1}^N\ell_i|U_i|=\sum_{i=1}^N\sum_{j=0}^{M_\varepsilon}\ell_i|W_{ij}|,\qquad \qquad \ell_i\in \N,
 \end{align}
 for a suitable  $(\ell_1,\dots,\ell_N)\in \text{{\rm Ad}}(\gamma_{(N)})$. From \eqref{45} we have
 \begin{align}\label{aN2}
 	\mathcal A^N(\gamma)\geq -C\varepsilon+ \sum_{j=0}^{M_\varepsilon}\sum_{i=1}^N(\ell_i+m_{ij})|W_{ij}|=-C\varepsilon+ \sum_{j=0}^{M_\varepsilon}\sum_{i=1}^N\widehat k_{ij}|W_{ij}|.
 \end{align}
where we have noted  $\widehat k_{ij}:=\ell_i+m_{ij}$; now we claim that 
\begin{align}\label{claimfinale*}
	\mathcal A(\gamma_\varepsilon)	\leq C\varepsilon+\sum_{j=0}^{M_\varepsilon}\sum_{i=1}^N\widehat k_{ij}|W_{ij}|.
\end{align}
This will be shown in the next step. 
%
%
%
%
From \eqref{claimfinale*} and \eqref{aN2} we will deduce
\begin{align*}
\mathcal A(\gamma_\varepsilon)\leq 2C\varepsilon+\mathcal A^N(\gamma),
\end{align*}
for all $N$ big enough.
But, from (b) and (c) in Lemma \ref{lem_approxcurves} 
we also have $\mathcal A(\gamma)\leq C\varepsilon+\mathcal A(\gamma_\varepsilon)$, thus we infer
\begin{align*}
	\mathcal A(\gamma)\leq 3C\varepsilon+\mathcal A^N(\gamma)\leq 3C\varepsilon+\lim_{N\rightarrow \infty}\mathcal A^N(\gamma), 
\end{align*}
which implies \eqref{reverse} by arbitraryness of $\varepsilon>0$. The thesis follows.

\textit{Step 3:}
We are left with proving \eqref{claimfinale*}. Using \eqref{Aeps}, this is equivalent to show
\begin{align}\label{Aeps2}
	\sum_{j=1}^{M_\varepsilon}k_j|V_j|=\sum_{j=1}^{M_\varepsilon}\sum_{i=0}^\infty k_j|W_{ij}|\leq C\varepsilon+\sum_{j=0}^{M_\varepsilon}\sum_{i=1}^N\widehat k_{ij}|W_{ij}|.
\end{align}
Setting $\widehat k_{0j}:=m_{0j}=m_\varepsilon(R_{0j})$, for all $j=1,\dots,M_\varepsilon$, still from \eqref{Ce} (and arguing as in \eqref{45}), we know that  $\sum_{j=1}^{M_\varepsilon}\widehat k_{0j}|W_{0j}|\leq C\varepsilon.$
Hence, using this, in order to show \eqref{Aeps2} we will prove that 
\begin{align}\label{claim_step2}
	\sum_{j=1}^{M_\varepsilon} k_j|V_{j}|\leq C\varepsilon+\sum_{j=0}^{M_\varepsilon}\sum_{i=0}^N\widehat k_{ij}|W_{ij}|
\end{align} 
Recalling that $\widetilde W_{(N+1)j}= \cup_{h> N}{W_{hj}}$, we set $$\widehat{k}_{(N+1)j}:=\min_{i>N}\{m_{ij}+k_j\}=m_{h_jj}+k_j=m_\varepsilon(R_{h_jj})+k_j$$ (for some index $h_j>N$) and we infer from \eqref{45}
$$\sum_{j=0}^{M_\varepsilon}\widehat k_{(N+1)j} |\widetilde W_{(N+1)j}|\leq C\varepsilon.$$ 
Whence, denoting $\widetilde W_{(N+1)j}$ by $W_{(N+1)j}$ (to shortcut notation), \eqref{claim_step2} will follow if we prove that, for some constant $C>0$, it holds
\begin{align}\label{claim_step3}
	\sum_{j=1}^{M_\varepsilon} k_j|V_{j}|\leq C\varepsilon+\sum_{j=0}^{M_\varepsilon}\sum_{i=0}^{N+1}\widehat k_{ij}|W_{ij}|.
\end{align} 
Denoting $R_{(N+1)j}:=R_{h_jj}$, we now consider the family $\mathcal R:=\{R_{ij}:i=0,\dots,N+1,\;j=1,\dots,M_\varepsilon\}$. We now see that the claim \eqref{claim_step3} follows from Proposition \ref{prop:crucial} provided that 
\begin{align}\label{413}
	(\dots\widehat k_{ij}\dots)\in \textrm{Ad}(\gamma_\varepsilon^{\mathcal R}),
\end{align}
because in this case, setting $k'_j:=\min_{i\leq N+1}\{\widehat k_{ij}\}$, we have
\begin{align*}
	\sum_{j=1}^{M_\varepsilon} k_j|V_{j}|\leq\sum_{j=1}^{M_\varepsilon} k'_j|V_{j}|= \sum_{j=1}^{M_\varepsilon} \sum_{i=0}^{N+1} k_j'|W_{ij}|\leq  \sum_{j=1}^{M_\varepsilon}\sum_{i=0}^{N+1}\widehat k_{ij}|W_{ij}|,
\end{align*} 
which implies \eqref{claim_step3}. Now, using the definition of $\widehat k_{ij}$, the fact that $R_{ij}\in V_j$ for all $i=0,\dots, N+1$, and so defining $P_j$ as one of them, \eqref{413} readily follows from Lemma \ref{Rem_34}. 
\end{proof}
\textbf{Data availability statement:} data sharing is not applicable to this article as no data were created or analyzed in this study.\\

  \textbf{Acknowledgements:}
MC thanks the financial support of PRIN 2022R537CS "Nodal optimization, nonlinear elliptic equations, nonlocal geometric problems, with a focus on regularity" funded by the European Union under Next Generation EU.  RS is  member of the Gruppo Nazionale per l'Analisi Matematica, la Probabilit\`a e le loro Applicazioni (GNAMPA) of the Istituto Nazionale di Alta Matematica (INdAM), and joins the project CUP\_E53C22001930001.
We thank the financial support of the F-CUR project number  2262-2022-SR-CONRICMIUR\_PC-FCUR2022\_002   of the University of Siena, and the financial support of 
PRIN 2022PJ9EFL "Geometric Measure Theory: Structure of Singular Measures, Regularity Theory and Applications in the Calculus of Variations" funded by the European Union under Next Generation EU. Views and opinions expressed are however those of the author(s) only and do not necessarily reflect those of the European Union or The European Research Executive Agency.  Neither the European Union nor the granting authority can be held responsible for them.

\bibliography{references}
\bibliographystyle{plain}

%
%
%
%
%


\end{document}